\documentclass[11pt,twoside, reqno]{article}

\usepackage{amsfonts}
\usepackage{amssymb}
\usepackage{amsmath}
\usepackage{amsthm}
\usepackage{latexsym}
\usepackage{mathrsfs}
\usepackage{bbm}
\usepackage{comment}

\usepackage{txfonts}
\usepackage{enumerate}
\usepackage{hyperref}
\usepackage{graphicx}
\usepackage{authblk}

\usepackage{color}
\usepackage{pgf,tikz}
\usepackage{mathrsfs}
\usetikzlibrary{arrows}
\allowdisplaybreaks

\usepackage{indentfirst}

\def\fz{\infty}

\def\r{\right}
\def\lf{\left}

\newcommand\SF{\mathscr F}
\newcommand\SM{\mathscr M}
\def\az{\alpha}
\def\lz{\lambda}

\def\dz {\delta}

\def\epz{\sigma}

\def\bz{\beta}

\def\fai{\varphi}
\def\gz{{\gamma}}

\def\sz{\sigma}

\def\wz{\widetilde}

\def\ls{\lesssim}

\def\dsum{\displaystyle\sum}
\def\dint{\displaystyle\int}
\def\dfrac{\displaystyle\frac}
\def\dsup{\displaystyle\sup}
\def\hs{\hspace{0.3cm}}

\def\e{{\rm e}}
\def\mathpal#1{\mathop{\mathchoice{\text{\rm #1}}%
   {\text{\rm #1}}{\text{\rm #1}}%
   {\text{\rm #1}}}\nolimits}
\newcommand\Hess{\mathpal{Hess}}
\newcommand\Ric{\mathpal{Ric}}

\renewcommand\div{\mathpal{div}}
\newcommand\tr{\mathpal{tr}}
\newcommand\HS{{\mathop{\scriptscriptstyle\rm HS}\nolimits}}
\newcommand\mequal{\overset{\text{\tiny m}}{=}}

\newcommand\E{\mathbb{E}}
\newcommand\R{\mathbb{R}}
\newcommand\N{\mathbb{N}}
\newcommand\Z{\mathbb{Z}}

\newtheorem{theorem}{Theorem}[section]
\newtheorem{lemma}[theorem]{Lemma}
\newtheorem{corollary}[theorem]{Corollary}
\newtheorem{proposition}[theorem]{Proposition}
\theoremstyle{definition}
\newtheorem{remark}[theorem]{Remark}

\numberwithin{equation}{section}
\def\supp{{\mathop\mathrm{\,supp\,}}}

\numberwithin{equation}{section}

\begin{document}

\title{\vskip-2.1cm\bf\Large Hessian heat kernel estimates and
  Calder\'{o}n-Zygmund inequalities\\ on complete Riemannian manifolds
  \footnotetext{\hspace{-0.35cm} * Corresponding author\endgraf
    \hspace{-0.35cm} 2010 {\it Mathematics Subject
      Classification}. Primary: 35K08; Secondary: 35J10, 35J30, 47G40.
    \endgraf {\it Key words and phrases}.  Brownian motion, Heat
    semigroup, Heat kernel, Bismut type Hessian representations,
    Hessian formula, Calder\'{o}n-Zygmund inequalities, the
    Hardy-Littlewood maximal function, good-$\lambda$ inequalities,
    Kato class.  \endgraf This work has been project is supported by
    the National Natural Science Foundation of China (Grant
    No. 12071431). The third author has been supported by the Fonds
    National de la Recherche Luxembourg (project GEOMREV
    O14/7628746).}}

\author[1]{Jun Cao} \author[2]{Li-Juan Cheng*} \author[3]{Anton Thalmaier}

\affil[1]{\small Department of Applied Mathematics, Zhejiang
  University of Technology,\par
  Hangzhou 310023, The People's Republic of China\par
  \texttt{caojun1860@zjut.edu.cn} \vspace{1em}
  }
 \affil[2]{\small School of  Mathematics, Hangzhou Normal University\par
    Hangzhou 311123, The People's Republic of China\par
    \texttt{chenglj@zjut.edu.cn}\vspace{1em}}
\affil[3]{\small
  Department of Mathematics, University of Luxembourg, Maison du
  Nombre,\par
  L-4364 Esch-sur-Alzette, Luxembourg\par
  \texttt{anton.thalmaier@uni.lu}\vspace{1em}}

\date{\today}
\maketitle

\begin{abstract} {\noindent We address some fundamental questions
    about geometric analysis on Riemannian manifolds. The
    $L^p$-Calder\'on-Zygmund inequality is one of the cornerstones in
    the regularity theory of elliptic equations, and it has been
    asked under which geometric conditions it holds for a reasonable class
    of non-compact Riemannian manifolds, and to what extent
    assumptions on the derivative of curvature and on the
    injectivity radius of the manifold are necessary.
    In the present paper, for $1<p<2$, we give
    a positive answer for the validity of the $L^p$-Calder\'on-Zygmund
    inequality on a Riemannian manifold assuming only a lower
    bound on the Ricci curvature. It is well known that  this  alone is
    not sufficient for $p>2$.
    In this case we complement the study of G\"uneysu-Pigola (2015) and
    derive sufficient geometric criteria for the validity of the
    Calder\'on-Zygmund inequality under additional Kato class bounds
    on the Riemann curvature tensor and the covariant derivative of
    Ricci curvature. Bounds in the Kato class are
    integral conditions and much weaker than pointwise
    bounds. Throughout the proofs, probabilistic tools, like Hessian
    formulas and Bismut type representations for heat semigroups, play
    a significant role.}
\end{abstract}

{\small\tableofcontents}

\section{Introduction\label{s1}}

The Hessian operator $\Hess$, which contains all information of the second order derivatives, is a fundamental object in the second order smooth analysis.
To study the Hessian operator, one usually needs 
the Calder\'{o}n-Zygmund inequality that controls 
$\Hess$ by the simpler Laplacian operator $\Delta$ 
(see \cite{CZ-52, GTbook,NST,GP-15, Pigola} and the references therein for its wide applications particularly in the regularity theory of elliptic equations). 

In the Euclidean space $\mathbb{R}^d$, the Calder\'on-Zygmund inequality says that 
for any $p\in (1,\infty)$ and $u\in C_c^{\infty}(\R^n)$, it holds 
\begin{align}\label{eqn-CZes}
	\|\Hess u\|_{L^p(\R^d)}\leq C\| \Delta u\|_{L^p(\R^d)},
\end{align} 
where $C=C(d,p)>0$ is a constant depending only on $d$ and $p$.
The inequality \eqref{eqn-CZes} is first proved by Calder\'on and Zygmund
via their seminal 
theory of singular integral operators based on the explicit 
representation of the Green kernel of the Laplacian.
A remarkable consequence of inequality \eqref{eqn-CZes} is the fact that,
in the Euclidean space, the Sobolev norms 
\begin{align*}
  \|u\|_{W^{2,p}(\R^d)}&=\|u\|_{L^p(\R^d)}+\|\nabla u\|_{L^p(\R^d)}+\|\Hess u\|_{L^p(\R^d)},\\
  \|u\|_{\widetilde{W}^{2,p}(\R^d)}&=\|u\|_{L^p(\R^d)}+\|\Delta u\|_{L^p(\R^d)}
\end{align*}
are equivalent on $C_c^{\infty}(\R^n)$ (see \cite{Pigola}).

The  Calder\'{o}n-Zygmund inequality \eqref{eqn-CZes}
extends to second order uniformly elliptic operators $L=-\div A \nabla$ with variable coefficients $A$ on domains $\Omega\subset \R^d$ without any boundary conditions (see~e.g.~\cite{GTbook}). In this setting, we have the following local Calder\'on-Zygmund inequality that 
given any domains $\Omega_1\Subset  \Omega$,  $p \in (1,\infty)$ and
$u\in C_c^{\infty}(\Omega)$, it holds
\begin{align}\label{eqn-CZlv}
\|\Hess  u\|_{L^p(\Omega_1)}\leq C\Big(\|u\|_{L^p(\Omega)}+\|Lu\|_{L^p(\Omega)}\Big),
\end{align}
where $C=C(\Omega_1,\Omega,d,p,A)>0$ is a constant
depending on $\Omega_1, \Omega, p,d$ and the elliptic coefficients of $A$. The proof of \eqref{eqn-CZlv} is based on \eqref{eqn-CZes} and a perturbation argument of $A$. In particular, if $A$ is a constant matrix, 
then one can get rid of 
the term $\|u\|_{L^p(\Omega)}$ in \eqref{eqn-CZlv}.

A further step was taken by G\"{u}neysu and Pigola in \cite{GP-15} where they considered the following global Calder\'{o}n-Zygmund inequality 
on Riemannian manifolds $M$ of the form that 
for any $p\in (1,\infty)$ and $u\in C_c^\infty(M)$,
it holds  
$$\|\Hess u\|_{L^p(M)}\leq C_1 \|u\|_{L^p(M)}+C_2\|\Delta u\|_{L^p(M)}, \quad \eqno{{\bf CZ}(p)}$$
where $\Delta$ is the Laplace-Beltrami operator on $M$ and $C_1$,
$C_2$ are two positive constants.  It is known that such an inequality
{\bf CZ($p$)} can hold or fail depending on $p$ and the geometry of
$M$. We give a brief summary of the state of the art on this subject
(see \cite{Pigola} for a more detailed survey, as well as
\cite{GP-2019}).  For an example of a Riemannian manifold of positive
sectional curvature where {\bf CZ($p$)} fails for large values of $p$,
see \cite{Marini-Veronelli-2021}, as well as
\cite{Veronelli_et_al-2021} for a general result in this direction.

To begin with, we make some conventions on the notation.  Throughout
this paper, let $(M,g)$ be a complete non-compact connected
$d$-dimensional Riemannian manifold, $\nabla$ the Levi-Civita
covariant derivative, $\Hess=\nabla d$ the Hessian operator on
functions, and $\mu$ the Riemannian volume measure on~$M$.  We denote
by $|\cdot|$ the norm in the tangent space, and by $\|\cdot\|_p$ the
norm in $L^p(M,\mu)$ for $1\leq p\leq \infty$. The Laplace-Beltrami
operator $\Delta$ acting on functions and forms, is understood as
self-adjoint positive operator on $L^2(\mu)$.

If $p=2$ and there is a bound $\Ric\geq -K$ of the Ricci curvature of
$M$ for some constant $K>0$, then it is well-known that {\bf CZ($2$)}
is a straightforward consequence of Bochner's identity, see the
Appendix. The extension of {\bf CZ($p$)} from $p=2$ to an arbitrary
$p\in (1,\infty)$ is much more involved and an intriguing
problem. Inspired by the proof of \eqref{eqn-CZes}, a possible way to
establish ${\bf CZ}(p)$ is a similar potential theoretical approach that
represents $u$ via the Green kernel of the Laplace-Beltrami
operator~$\Delta$ (see \cite{NST,Wang2004,Ni}). Although
the method works pretty well for many related questions of the
associated Poisson equation, it has the drawback that it applies only
under some restrictive conditions such as compactness of $M$ or nonnegative Ricci
curvature.

To overcome this drawback, G\"{u}neysu and Pigola \cite{GP-15,Pigola}
introduced two methods of proof for {\bf CZ}$(p)$ avoiding the use of
Green kernel. The first method is based on a gluing
procedure that connects the local consideration to the global
result. To be precise,  if $M$ has bounded Ricci curvature and a
strictly positive injectivity radius, then it is proved in
\cite{GP-15} that the Calder\'on-Zygmund
inequality {\bf CZ}$(p)$ holds for any $p\in (1,\infty)$ with
implicit constants depending on $p$, $\|{\Ric}\|_{L^\infty}$, the
dimension $d$  and the injectivity radius. Moreover, if 
$p>\max\{2,~d/2\}$ and  $M$ has bounded sectional curvature, then {\bf CZ}$(p)$
also holds (see \cite[Theorem 5.18]{Pigola}). The second method,
called the functional analytic method, uses boundedness results for
the covariant Riesz transform for $1<p<2$ from \cite{ThW-04}.

Let us sketch the main idea of the second method in \cite{GP-15}.
Inequality {\bf CZ($p$)} is usually reduced to the existence of
positive constants $C$ and $\sigma$ such that
\begin{align*}
  \left\|\,|\Hess (\Delta+\sigma)^{-1}u|\,\right\|_p\leq C\,\|u\|_p,
\end{align*}
which is equivalent to
\begin{align}\label{eqn-equiv1}
  \left\|\,|\nabla (\Delta_1+\sigma)^{-1/2}
  \circ d(\Delta_0+\sigma)^{-1/2}u|\,\right\|_p\leq C\,\|u\|_p.
\end{align}
Here and hereafter, we write $\|\cdot\|_{p}:=\|\cdot\|_{L^p(M)}$ for simplicity.
The problem is thus reduced to the study of conditions for boundedness
of the classical Riesz transform $d(\Delta_0+\sigma)^{-1/2}$ on
functions and boundedness of the covariant Riesz transform
$\nabla (\Delta_1+\sigma)^{-1/2}$ on one-forms.  This approach in
\cite{GP-15} however is restricted to $p\in (1,2)$; the constants
$C_1, C_2$ in {\bf CZ($p$)} depend on dimension $d$, $p$, $\|R\|_{\infty}$,
$\|\nabla R\|_{\infty}$ and on the constants $D, \delta$ from the
following local volume doubling assumption: there are constants $C>0$,
$0\leq \delta < 2$ such that
\begin{align}\label{eqn-GD}
  V(x,tr)\leq Ct^d\e^{Ctr}V(x,r) \tag{\bf LD}
\end{align}
for all $x\in M$, $r>0$ and $t\geq 1$, where $V(x,r):=\mu(B(x,r))$.

Very recently, Baumgarth, Devyver and G\"{u}neysu \cite{BDG-21}
studied the covariant Riesz transform on $j$-forms. Their results can
be applied to {\bf CZ($p$)} when $1<p<2$ requiring the same curvature
conditions as in \cite{GP-15} but without the local volume doubling
assumption made in \cite{GP-15}. This comes from the fact that
\eqref{eqn-GD} already holds if $\Ric\geq -K$ for some $K\geq 0$, as
can be seen by the Bishop-Gromov comparison theorem and the well-known
formula for the volume of balls in hyperbolic space.  However, it
seems difficult to establish {\bf CZ($p$)} for $p>2$ in this way, as
when trying to extend the machinery of \cite{PTTS-2004} to
$L^p$-boundedness of covariant Riesz transform for $p>2$, the local
Poincar\'{e} inequality, used explicitly in \cite{PTTS-2004}, does not
make sense on differential forms.

\smallskip

The observations above raise the following questions:
\begin{enumerate}[1.]
\item\emph{In the case $1<p<2$, is it possible to weaken the
    assumptions on the Riemann curvature tensor $\|R\|_{\infty}$ and
    $\|\nabla R\|_{\infty}$, e.g.,~replacing them by Ricci curvature
    bounds?}
\item\emph{Without a lower control on the injectivity radius, under
    which conditions on the manifold $M$, the inequality {\bf CZ($p$)}
    holds on $M$ for $p>2$?}

\end{enumerate}

In order to answer the above two questions affirmatively, we develop a
new functional analytic method of proof for {\bf CZ}$(p)$ that works
for all $p\in (1,\infty)$. Unlike the second method used in
\cite{GP-15,Pigola} that reduces {\bf CZ($p$)} to the boundedness
\eqref{eqn-equiv1}, our method makes use of the observation that
{\bf CZ($p$)} is equivalent to the $L^p$-boundedness of the following
operator:
\begin{align*}
\Hess(\Delta+\sigma)^{-1}=\int_0^{\infty}\e^{-\sigma t}\Hess P_t\,dt,
\end{align*}
where $P_t$ denotes the heat semigroup generated by $-\Delta$. Based on this observation, our strategy 
to solve the above two questions is first establishing some Hessian heat kernel estimates and then bridging from Hessian heat kernel estimates to {\bf CZ($p$)}. 

However, references on Hessian heat kernel estimates are sparsely to find in the literature,
especially when compared to the situation of heat kernel estimates and
gradient estimates for the heat kernel. Therefore, in the following section, we invest some
effort in deriving sharp Hessian heat kernel estimates first. 
It turns out that there is a big difference
between the treatment of the cases $p<2$ and $p>2$.
In particular, if $p\in (1,2)$, we only assume that $M$ satisfies the following curvature~condition: 
\begin{align}\label{eqn-Ric}
	\Ric\geq -K    \tag{\bf Ric}
\end{align}
for some $K\geq 0$. The curvature condition \eqref{eqn-Ric} then give rise to the local Gaussian upper bounds of the heat kernel and the local volume doubling condition {\bf (LD)} which enable us to deduce a series of 
$L^2$ weighted off-diagonal estimates of the Hessian heat kernel (see Section \ref{s2.1} below). These estimates
together with the classical argument of Calder\'on-Zygmund decomposition (see \cite{TX-99}) allow to derive our first main result.

\begin{theorem}\label{them-1}
	Let $(M,g)$ be a complete Riemannian manifold satisfying
	\eqref{eqn-Ric}. Let $1<p<2$ be fixed. Then there exists a constant
	$\sigma>0$ such that the operator $\Hess (\Delta+\sigma)^{-1}$ is
	bounded in $L^p$, i.e.~{\bf CZ($p$)} holds.
\end{theorem}

Theorem \ref{them-1} answers question $1$ affirmatively. Comparing Theorem \ref{them-1} with existing results on {\bf CZ($p$)},
it should be pointed out first that the result is valid without any
injectivity radius assumptions and secondly, instead of boundedness of
$\|R\|_{\infty}$ and $\|\nabla R\|_{\infty}$, only a lower bound of
$\Ric$ is needed.  In view of the argument for $p=2$, the assumption \eqref{eqn-Ric} seems close to 
sharp.

%
%
%
%

%

%
%

For the case $p>2$,  it is well known
\cite{Pigola,Marini-Veronelli-2021,Veronelli_et_al-2021} that the curvature condition \eqref{eqn-Ric}
alone is not enough for {\bf CZ($p$)} to hold.  
Hence more geometric information about the manifold is
required in this case for the validity of {\bf CZ($p$)}.

To formulate appropriate geometric conditions, we introduce some  probabilistic quantities.
Denote by $X_t$ the diffusion process generated by $-\Delta$, which is
assumed to be non-explosive ($M$ is stochastically complete). Then,
for $f\in \mathcal{B}_b(M)$,
\begin{align*}
  \E^x[f(X_t)]=P_tf(x)=(\e^{-\Delta t}f)(x).
\end{align*}
A Borel function $K\colon M\rightarrow \R$ is said to be in the {\it Kato
class} $\mathcal{K}(M)$ of $M$, if
\begin{align}\label{eqn-Katoc}
\lim_{t\rightarrow 0+}\sup_{x\in M}\int_0^t \E^x\big[|K(X_s)|\big]\,ds=0. \tag{{\bf K}}
\end{align}
Obviously, $\mathcal{K}(M)$ is a linear space and
$\mathcal{K}(M)\subset L_{\rm loc}^1(M)$.
The Kato class has been introduced in \cite{Kato} on Euclidean space and then applied 
to investigate singular potentials, see for instance
\cite{AS-82,Simon-84,Gu-12,rose_stollmann_2020}.  
Concerning criteria for functions to be in the Kato class the reader
may consult \cite{Guen-2017,Gueneysu-book-2017, GP13}.  Note that for dimension $d\geq 2$, if \eqref{eqn-Ric} holds, then the heat kernel has a local on-diagonal estimate which implies $L^p(M)+L^{\infty}(M)\subset \mathcal{K}(M)$ for $p>d/2$ by \cite[Proposition 3.2]{GP13}.

To deal with the case $p>2$, 
our conditions are 
given in terms of Kato bounds on the  geometric quantities.

\smallskip 

\noindent {\bf Condition (H) } \emph{ $\Ric\geq -K$ for some $K\geq 0$
	and there exist $K_1,K_2\in \mathcal{K}(M)$ such that
	\begin{align}\label{eqn-H}
		|R|^2(x)\leq K_1(x)\quad\text{and}\quad  \big|\nabla \Ric^{\sharp}+d^*R\big|^2(x)\leq K_2(x),\quad
		x\in M, \tag{\bf H} 
\end{align}
where  for $x\in M$ and $ v_1,v_2,v_3\in T_xM$,
\begin{align*}
	&|R|(x)=\sup \Big\{|R^{\#,\#}(v_1,v_2)|_\HS(x)\colon\ v_1,v_2\in T_xM,\ |v_1|\leq 1,  \ |v_2|\leq 1 \Big\},
\end{align*}
and
\begin{align*}
	& \langle d^*R(v_1,v_2), v_3 \rangle=\langle (\nabla_{v_3}\Ric^{\sharp})(v_1), v_2 \rangle-
	\langle (\nabla_{v_2}\Ric^{\sharp})(v_3), v_1 \rangle,
\end{align*}
 with $R^{\#,\#}(v_1,v_2)=R(\cdot,v_1,v_2,\cdot)$, the  curvature tensor $R$ and $\Ric^{\sharp}(v)=\Ric(\cdot, v)^\sharp$ for $v\in T_xM$.}

\medskip

Under the condition~${\bf (H)}$, we are able to prove the 
following key pointwise inequalities for the Hessian of the semigroup that for any $t>0$,
$f\in C_c^\infty(M)$ and $x\in M$, it holds 
\begin{align}\label{eqn-pihe}
	|\Hess P_tf|(x)\leq \e^{2Kt}P_t|\Hess f|(x) + C\e^{(2K+\theta)t}(P_{t}|\nabla f|^2)^{1/2}(x),
\end{align}
and
\begin{align}\label{eqn-Bismut-He}
 &t\,|\Hess P_tf|(x)\leq C(1+\sqrt{t})\,\e^{(2K+\theta)t}\lf(P_t|f|^2\r)^{1/2}(x)
\end{align}
for some constants $C,\theta>0$ (see Propositions \ref{lem1} and
\ref{lem5} below).  Both of them are proved by using some
probabilistic tools established in
\cite{EL94,EL-98,Li,Stroock,APT-03}. More precisely, inequality
\eqref{eqn-pihe} is proved by a stochastic approach based on a second
order derivative formula for the heat semigroup (see
\eqref{Hessian-formula} below); inequality \eqref{eqn-Bismut-He} is
established by means of Bismut-type representation formulas for the
Hessian of heat semigroups which were first proved by Elworthy and Li
(see \cite{EL94,EL-98}).

The inequalities \eqref{eqn-pihe} and \eqref{eqn-Bismut-He} enable us to
establish a series of pointwise Hessian heat kernel estimates.
Based on these results, we can give an affirmative answer to the question $2$ above.

\begin{theorem}\label{theorem-2}
	Let $M$ be a complete Riemannian manifold satisfying
	\eqref{eqn-H}. Let $p>2$ be fixed. Then there exists a constant
	$\sigma>0$ such that the operator $\Hess\,(\Delta+\sigma)^{-1}$ is
	bounded in $L^p$, i.e.~\text{\bf CZ}$(p)$ holds.
\end{theorem}

Theorem \ref{theorem-2} gives in particular an answer to the open
question in \cite{GP-15} about sufficient conditions for {\bf CZ($p$)}
when $p>2$ in the absence of control of the injectivity radius. It is
worth mentioning that compared with the sufficient conditions even in
the case $1<p<2$ (\cite[Theorem~D]{GP-15},
\cite[Corollary~1.8]{BDG-21}), in Theorem \ref{theorem-2} the
geometric quantities $|R|$ and $|\nabla \Ric|$ do not need to be
uniformly bounded on $M$; it is sufficient to have bounds in the Kato
class which is a kind of integral condition. 
It should be mentioned that for $p>d/2$ in conjunction with
$p\geq2$ the validity of \text{\bf CZ}$(p)$ has been obtained under
strong curvature assumption, namely uniform boundedness of the
sectional curvature tensor, but without conditions involving the
derivative of curvature, see Theorem 5.18 in the survey of
Pigola~\cite{Pigola}.

As indicated above, the pointwise inequalities \eqref{eqn-pihe} and
\eqref{eqn-Bismut-He} play a key role in the proof of Theorem
\ref{theorem-2}, which also reflects the difference between our
approach and that of classical local Riesz transform (see
\cite{CD-01,CD-03, PTTS-2004}).  In the latter case, the following
domination property that for some positive constants $c_1,c_2$ and $C$
\begin{align}\label{Pt-one}
  |\nabla P_tf|\leq C\e^{c_1t}P_{c_2t}|\nabla f|
\end{align}
is indispensable for boundedness of the
Riesz transform in case $p>2$ (see \cite{CD-01,CD-03}).
Unfortunately, the machinery of \cite{CD-01} cannot work well in our situation since a
Hessian estimate of the type
\begin{align*}
  |\Hess P_tf|\leq C\e^{c_1t}P_{c_2t}|\Hess f|
\end{align*}
would be required which however only holds in very specific cases (like flat
manifolds). On the other hand, if we aim at using the techniques from
\cite{PTTS-2004} directly, the main difficulty to deal with is that there is no
suitable Hessian replacement of the local Poincar\'e inequality which is
heavily used throughout their proof.

To overcome these two obstacles, we take advantage of the techniques
from \cite{PTTS-2004}, in particular, the sharp maximal function and
good-$\lambda$ inequalities. By means of these tools and inequality
\eqref{eqn-pihe}, we observe that if the additional term involving
$(P_t|\nabla f|^2)^{1/2}$ is treated and controlled as an error term, then
boundedness of $\Hess (\Delta+\sigma)$ in $L^p$ can be
established also for $p>2$ with the help of pointwise Hessian estimates
of the heat kernel.  In conclusion, the crucial observation is that the pointwise inequality
\eqref{eqn-pihe} may serve as a Hessian replacement of
\eqref{Pt-one} circumventing the non-availability
of the local Poincar\'e inequality.

The rest of the paper is organized as follows. In Section 2, we give
various forms of estimates on the Hessian of the heat kernel and on
the corresponding semigroups, which are used throughout the proof of
{\bf CZ($p$)}. In Sections 3 finally, we present proofs for Theorem
\ref{them-1} and \ref{theorem-2} respectively. A suitable version of
the localization techniques of \cite{DR-96} is included.

\medskip{\bf Acknowledgements} The authors are indebted to Batu
G\"uneysu, Stefano Pigola and Giona Veronelli for very helpful
comments on the first draft of this paper.

\section{Hessian heat kernel  estimates }
This section is divided into two parts: the first part is on
$L^2$-estimates of the Hessian of heat kernels under the assumption
of a lower Ricci curvature bound; the second part is on the estimates
derived under condition {\bf (H)} from the Hessian formula and Bismut
type formulas for the Hessian of semigroups, which are established by
means of the techniques from stochastic analysis.

\subsection{$L^2$-estimates for the Hessian of heat kernel}\label{s2.1}

In this section, we assume Ricci curvature to be bounded below,
i.e.~validity of condition {(\bf Ric)}.  Then, in particular, the
local doubling assumption \eqref{eqn-GD} with respect to $\mu$ holds.
Then for $x$ and $y\in M$, we obviously have
$B(y,\!\sqrt{t})\subset B(x,\sqrt{t}+\rho(x,y))$. Thus if
\eqref{eqn-GD} holds, then
\begin{align}\label{eqn-gdxy}
  V(y,\!\sqrt{t})\leq V\Big(x,\sqrt{t}+\rho(x,y)\Big)
  \leq C\Big(1+\frac{\rho(x,y)}{\sqrt{t}}\Big)^{d}\exp\left(C(\sqrt{t}
  +\rho(x,y))\right)V(x,\sqrt{t}).
\end{align}

It is well-known that, under a lower Ricci curvature bound, the heat
kernel allows an off-diagonal estimate \cite{Davies93}.  The following
lemma gives a pointwise off-diagonal estimate for the heat kernel and
its time derivative.
\begin{lemma}\label{lem-Gausspe}
  Assume that \eqref{eqn-Ric} holds.  Then for any
  $\az\in (0,\frac{1}{4})$, there exist constants $C$ and $C_1>0$
  depending on the dimension $d$ and $\alpha$ such that for all
  $x,y\in M$ and $t>0$,
  \begin{align}\label{heat-kernel-upper-bound}
    p_t(x,y)+\left|\frac{\partial p_t}{\partial t}(x,y)\right| \leq \frac{C}{V(y,\!\sqrt{t})}
    \exp\left(-\alpha \frac{\rho^2(x,y)}{t}+C_1 Kt\right).
  \end{align}
\end{lemma}

\begin{proof}
  The estimate of $p_t(x,y)$ is an easy consequence of
  \eqref{eqn-gdxy} and \cite[Theorem 2]{Davies93} or \cite[Theorem
  2.4.4]{Wbook14}, where it is proved that for all $x,y\in M$ and
  $t>0$,
  \begin{align}\label{heat-kernel-upper-bound-2}
    p_t(x,y)\leq \frac{C}{V(y,\!\sqrt{t})}
    \exp\lf(-\alpha \frac{\rho^2(x,y)}{t}+C_1 Kt\r).
  \end{align}
  The estimate for $|\frac{\partial p_t}{\partial t}(x,y)|$ follows
  from that of $p_t(x,y)$ and the analytic property of the semigroup
  (see \cite[Theorem 4]{Davies97} or \cite[Corollary 3.3]{Gr97}).
\end{proof}

The following lemma gives weighted $L^2$-integral estimates for the
heat kernel, its gradient and its Laplacian.

\begin{lemma}\label{lem4}
  Assume that \eqref{eqn-Ric} holds.  Let $\az\in (0,\frac{1}{4})$ be
  as in Lemma \textup{\ref{lem-Gausspe}}.  For all
  $\gamma \in (0,2\alpha)$, $s>0$ and $y\in M$,
  \begin{align*}
    \int_M \lf[|p_s(x,y)|^2+s|\nabla _x p_s(x,y)|^2+s^2|\Delta_x p_s(x,y)|^2\r]
    \e^{\gamma\, \frac{\rho^2(x,y)}{s}}\,\mu(dx)
    \leq \frac{C_{\gamma}}{V(y,\!\sqrt{s})}\,\e^{2\,C's},
  \end{align*}
  where $C_{\gamma}>0$ depends on $\gamma$ and $C'>0$ on $\az$ and
  $K$.
 	
\end{lemma}
\begin{proof}
  By \eqref{eqn-GD}, it is easy to see that for all $\gamma>0$,
  $s,t>0$ and $y\in M$, there exist two positive constants
  $C_{\gamma}$ (depending on $\gamma$ and the constants in
  \eqref{eqn-GD}) and $\tilde{C}$ such that
  \begin{align}\label{eqn-ele1}
    \int_{\rho(x,y)\geq \sqrt{t}}\e^{-2\gamma\frac{\rho^2(x,y)}{s}}\,\mu(d x)
    &\leq \e^{-\gamma t/s}\int_M \e^{-\gamma \frac{\rho^2(x,y)}{s}}\,\mu(dx) \notag\\
    &\leq  \e^{-\gamma t/s}\sum_{i=0}^{\infty}V(y,(i+1)\sqrt{s})\e^{-\gamma i^2}\notag\\
    &\leq C  \e^{-\gamma t/s}V(y,\!\sqrt{s}) \sum_{i=0}^{\infty}
      (i+1)^d \e^{-\gamma i^2}\e^{C(i+1)\sqrt{s}} \notag\\
    &\leq  C  \e^{-\gamma t/s}\e^{C\sqrt{s}}V(y,\!\sqrt{s}) \sum_{i=0}^{\infty}
      (i+1)^d \e^{-\gamma i^2}\e^{\gamma i^2/2+C^2s/2}\notag\\
    & \leq  C  \e^{-\gamma t/s}\e^{C\sqrt{s}+C^2s/2}V(y,\!\sqrt{s}) \sum_{i=0}^{\infty}
      (i+1)^d \e^{-\gamma i^2/2}\notag\\
    & \leq C_{\gamma} V(y,\!\sqrt{s}) \,\e^{-\gamma t/s} \e^{\tilde{C}s},
  \end{align}
  where the second inequality comes from condition \eqref{eqn-GD}. The
  remainder of the proof follows from \cite[Lemmas 2.1-2.3]{TX-99}.
\end{proof}
 
We now turn to the estimates for the Hessian of heat kernel. The
following lemma shows that the Hessian of heat semigroup also
satisfies an $L^2$-Gaffney off-diagonal estimate. Note that in the
following discussion the constant $C$ will be different in different
lines without confusion.

\begin{lemma}\label{lem-Gaffney}
  Assume that \eqref{eqn-Ric} holds.  There exist constants $C,C_2>0$
  such that for all $t\in (0,\fz)$, all Borel subsets $E,F\subset M$
  with compact closure, and all $f\in L^2(M)$ with {\rm
    supp}\,$f\subset E$,
  \begin{align*}
    \big\|\mathbbm{1}_Ft\, |\Hess P_tf|\big\|_2\leq C(1+\sqrt{t})\,\exp\left(-\frac{C_2\,\rho^2(E,F)}{t}\right)\|f\|_2.
  \end{align*}
\end{lemma}

\begin{proof}
  Recall the following $L^2$-Gaffney off-diagonal estimate on
  one-forms \cite{BDG-21}. If \eqref{eqn-Ric} holds, then for
  $\alpha \in \Gamma_{L^2}(T^*M)$ with support
  ${\rm supp}(\alpha)\subset E$ and any $s\in (0,1)$, it holds
  \begin{align*}
    \big\|\mathbbm{1}_F\!\sqrt{s}\,|\nabla \e^{-s\Delta^{(1)}}\alpha|\big\|_2
    &\leq C\,(1+\sqrt{s})\,\exp\left(-\frac{c_1\,\rho^2(E,F)}{s}\right)\|\mathbbm{1}_E \alpha\|_2\\
    &\leq 2C\exp\left(-\frac{c_1\rho^2(E,F)}{s}\right)\|\mathbbm{1}_E \alpha\|_2
  \end{align*}
  for some positive constants $C$ and $c_1$.  On the other hand, we
  have Gaffney's off-diagonal estimate for $|\nabla P_{t_2}f|$ (see
  \cite[(3.1)]{PTTS-2004}), i.e.~for all $f \in L^2(M)$ with support
  in $E$ and $u>0$,
  \begin{align*}
    \big\|\mathbbm{1}_{F}\!\sqrt{u}\,|\nabla P_{u}f|\big\|_2\leq C\exp\left(-\frac{c_2\rho^2(E,F)}{u}\right)\|\mathbbm{1}_E f\|_2
  \end{align*}
  for some positive constants $C$ and $c_2$.  For $t>0$, denoting by
  $\Delta^{(1)}$ the Laplacian on one-forms, we may write
 $$\Hess P_tf=\frac{\sqrt{\left(t-(\frac{t}{2}\wedge 1)\right)}\,\nabla \e^{-\big(t-(\frac{t}{2}\wedge 1)\big)\Delta^{(1)}}\left(\sqrt{\frac{t}{2}\wedge 1}\,d P_{\frac{t}{2}\wedge 1}f\right)}{\sqrt{\left(t-(\frac{t}{2}\wedge 1)\right)\big(\frac{t}{2}\wedge 1\big)}}  $$
 so that using the composition rule of Gaffney's off-diagonal estimate
 (see \cite[Lemma 2.3]{HM03}), we obtain
 \begin{align*}
   \big\|\mathbbm{1}_{E}t\, |\Hess P_{t}f|\big\|_2\leq C\frac{t}{\sqrt{\left(t-(\frac{t}{2}\wedge 1)\right)(\frac{t}{2}\wedge 1)}}\exp\left(-\frac{c\,\rho^2(E,F)}{\max\left\{\frac{t}{2}\wedge 1, t-(\frac{t}{2}\wedge 1)\right\}}\right)\|\mathbbm{1}_E f\|_2,
 \end{align*}
 for some positive constants $C$ and $c$.  Note that for $0<t/2<1$,
 \begin{align*}
   \max\left\{\frac{t}{2}\wedge 1, t-\big(\frac{t}{2}\wedge 1\big)\right\}=\max\left\{\frac{t}{2}, t-\frac{t}{2}\right\}=\frac{t}{2};\quad
   \frac{t}{\sqrt{\left(t-\big(\frac{t}{2}\wedge 1\big)\right)(\frac{t}{2}\wedge 1)}}=2,
 \end{align*}
 and that for $t/2\geq 1$,
 \begin{align*}
   \max\left\{\frac{t}{2}\wedge 1, t-\big(\frac{t}{2}\wedge 1\big)\right\}=\max\big\{1, t-1\big\}=t-1\leq t;\quad
   \frac{t}{\sqrt{\left(t-\big(\frac{t}{2}\wedge 1\big)\right)(\frac{t}{2}\wedge 1)}}=\frac{t}{\sqrt{t-1}}\leq \sqrt{2t}.
 \end{align*}
 Therefore, we conclude that there exist positive constants $C$ and
 $C_2$ such that
 \begin{equation*}
   \big\|\mathbbm{1}_{F}t\, |\Hess P_{t}f|\big\|_2\leq C\,(1+\sqrt{t})\,\exp\left(-\frac{C_2\rho^2(E,F)}{t}\right) \|\mathbbm{1}_E f\|_2.\qedhere
 \end{equation*}
\end{proof}

The following proposition gives $L^2$-weighted estimates for the
Hessian of the heat kernel.

\begin{proposition}\label{lem-Hess}
  Assume that \eqref{eqn-Ric} holds.  Fix $\az\in (0,\frac{1}{4})$ as
  in Lemma \textup{\ref{lem-Gausspe}}.  Then for all
  $\gamma\in (0,2\alpha)$, $s>0$ and $y\in M$, there exists a constant
  $C>0$ such that
  \begin{align*}
    \int_M |\Hess_xp_s(x,y)|^2\,\exp\left(\gamma\,\frac{\rho^2(x,y)}{s}\right)\,\mu(dx)\leq \frac{C(1+Ks)\,\e^{2C's}}{s^2 V(y,\!\sqrt{s})}
  \end{align*}
  where $C'>0$ is the same constant as in Lemma \ref{lem4}.
\end{proposition}

\begin{proof} 
  We begin by integrating Bochner's identity \eqref{eqn-BI} to obtain
  \begin{align*}
    \frac{1}{2}&\int |\nabla p_s|^2 \,\Delta \e^{\gamma\rho^2(x,y)/s}\,  \mu(dx)
                 =\frac{1}{2}\int_M (\Delta |\nabla p_s|^2)\,\e^{\gamma\rho^2(x,y)/s}\,\mu(dx)\\
               &=\int_M \left(-|\Hess p_s|^2_\HS+g(\nabla \Delta p_s, \nabla p_s )-\Ric(\nabla p_s,\nabla p_s)\right)\e^{\gamma\rho^2(x,y)/s}\,\mu(dx),
  \end{align*}
  which then implies
  \begin{align}\label{eq-Hess-upper-bound}
    &\int_M |\Hess p_s|_\HS^2\,\e^{\gamma\rho^2(x,y)/s}\,\mu(dx) \\
    &=-\frac{1}{2}\int_M |\nabla p_s|^2\Delta \e^{\gamma\rho^2(x,y)/s}\,\mu(dx)
      +\int_Mg( \nabla \Delta p_s, \nabla p_s) \,\e^{\gamma\rho^2(x,y)/s}\,\mu(dx)\notag\\
    &\quad-\int_M 
      \Ric (\nabla p_s, \nabla p_s)\,\e^{\gamma\rho^2(x,y)/s}\,\mu(dx)\notag\\
    &=-\frac{\gamma}{s}\int_M |\nabla p_s|^2\div\left(\e^{\gamma\rho^2(x,y)/s}\rho(x,y) 
      {\nabla}\rho(\cdot, y)(x)\right)\,\mu(dx) + \int \Delta p_s\div (\e^{\gamma\rho^2(x,y)/s}\nabla p_s)\,\mu(dx)\notag\\
    &\quad-\int \Ric(\nabla p_s, \nabla p_s)\,\e^{\gamma\rho^2(x,y)/s}\,\mu(dx)
      =:\mathrm{I}_1+\mathrm{I}_2+\mathrm{I}_3 \notag
  \end{align}
  (with the sign convention $\Delta=\div\nabla$ for the
  divergence). Now for any $\beta\in (0,2\az)$, let
  \begin{align*}
    E(s,y,\beta):=\int_M \lf[|p_s(x,y)|^2+s|\nabla _x p_s(x,y)|^2+s^2|\Delta_x p_s(x,y)|^2|\r]
    \e^{\beta\rho^2(x,y)/s}\,\mu(dx).
  \end{align*}
  Then by condition \eqref{eqn-Ric} and Lemma \ref{lem4}, we have
  \begin{align}\label{eqn-esI3}
    \mathrm{I}_3&\le K\int_M |\nabla p_s|^2\,\e^{\gamma\rho^2(x,y)/s}\,\mu(dx)
                  \le \frac{K}{s} E(s,y,\gz)\le \frac K s \frac{C_{\gamma}}{V(y,\!\sqrt{s})}\,\e^{2 C's}.
  \end{align}
 
  For $\mathrm{I}_2$, using the fact $|\nabla\rho|\le 1$, we deduce
  from Lemma \ref{lem4} again that
  \begin{align}\label{eqn-esI2}
    \mathrm{I}_2&=\int_M (\Delta p_s)^2\,\e^{\gamma\rho^2(x,y)/s}\,\mu(dx)
                  -\frac{2\gamma}{s}\int_M (\Delta p_s)\,\e^{\gamma\rho^2(x,y)/s} \rho(x,y)\langle\nabla \rho, \nabla p_s \rangle \,\mu(dx) \notag \\
                &\leq 2\int_M (\Delta p_s)^2\,\e^{\gamma\rho^2(x,y)/s}\,\mu(dx)+
                  \frac{\gamma^2}{s^2}\int_M \e^{\gamma\rho^2(x,y)/s}\rho^2(x,y)|\nabla p_s|^2\,\mu(dx)\notag\\ 
                &\le  \frac{C_\gz}{s^2}E(s,y,\gz')\le \frac 1 {s^2} \frac{C_{\gamma'}}{V(y,\!\sqrt{s})}\,\e^{2 C's},
  \end{align}
  where $\gamma<\gamma'<2\alpha$ with $\gamma'-\gamma$ sufficiently
  small.

  To bound $\mathrm{I}_1$, we write
  \begin{align}\label{eqn-esI1-1}
    \frac s{\gamma}\mathrm{I}_1
    &=\int_M |\nabla p_s|^2\,\e^{\gamma\rho^2(x,y)/s}\,\mu(dx) +2
      \int_M |\nabla p_s|^2\,\e^{\gamma\rho^2(x,y)/s}\frac{\gamma\rho^2(x,y)}{s}\,\mu(dx) \notag \\
    &\quad-\int_M |\nabla p_s|^2 \,\e^{\gamma\rho^2(x,y)/s}\rho(x,y)\Delta \rho(\cdot, y)(x)\,\mu(dx).
  \end{align}
  By the Laplacian comparison theorem (see e.g. \cite[p185]{Chavel}),
  we have
  \begin{align*}
    -\Delta \rho(\cdot,y)(x)\leq \frac{d-1}{\rho(x,y)}+K\rho(x,y)
  \end{align*}
  outside of the cut-locus. This yields in particular
  \begin{align*}
    &-\frac{\gamma}{s}\int_M |\nabla p_s|^2 \,\e^{\gamma\rho^2(x,y)/s}\rho(x,y)\Delta \rho(\cdot,y)(x)\,\mu(dx)\\
    &\leq \frac{ (d-1)\gamma}{s}\int_M |\nabla p_s|^2 \,\e^{\gamma\rho^2(x,y)/s}\,\mu(dx)+K\int_M |\nabla p_s|^2 \,\e^{\gamma\rho^2(x,y)/s}\,\frac{\gamma\rho^2(x,y)}{s}\,\mu(dx)\\
    &\leq  \frac{ (d-1)\gamma}{s}\int_M |\nabla p_s|^2 \,\e^{\gamma\rho^2(x,y)/s}\,\mu(dx)+K\int_M |\nabla p_s|^2 \,\e^{\frac{\gamma' \rho^2(x,y)}{s}}\,\mu(dx)\\
    &\le  \frac{C_\gz(1+Ks)}{s^2}E(s,y,\gz')
  \end{align*}
  with $\gamma<\gamma'<2\alpha$ as in \eqref{eqn-esI2}. Combining this
  estimate with \eqref{eqn-esI1-1}, it follows
  \begin{align}\label{eqn-esI1}
    \mathrm{I}_1\leq C_\gamma\frac{(1+Ks)}{s^2}E(s,y,\gz')
    \leq \frac{c(1+Ks)}{s^2\sqrt{V(y,\!\sqrt{s})}}\,\e^{2C' s}.
  \end{align}
  Altogether \eqref{eq-Hess-upper-bound} through \eqref{eqn-esI1}, we
  conclude that
  \begin{align*}
    \int_M |\Hess p_s|_\HS^2\,\e^{\gamma\rho^2(x,y)/s}\,\mu(dx)
    =\mathrm{I}_1+\mathrm{I}_2+\mathrm{I}_3 \leq c\frac{(1+Ks)}{s^2\sqrt{V(y,\!\sqrt{s})}}\,\e^{2C's}.
  \end{align*}
  which completes the proof of Proposition \ref{lem-Hess} by using the
  fact
  \begin{equation*}
    |\Hess p_s|^2\leq |\Hess p_s|^2_\HS.\qedhere
  \end{equation*}
\end{proof}

Using Proposition \ref{lem-Hess}, we obtain the following
$L^2$-integral estimates for Hessian of heat kernel.

\begin{corollary}\label{prop1}
  Assume \eqref{eqn-Ric} holds.  Fix $\az\in (0,\frac{1}{4})$ as in
  Lemma \textup{\ref{lem-Gausspe}}.  There exist $0<\beta<\alpha$ and
  $C''>C'>0$ with $C'$ as in Lemma \ref{lem4} such that
  \begin{align*}
    \int_{\rho(x,y)\geq t^{1/2}}|\Hess_xp_s(x,y)|\,\mu(dx)\leq 
    C\big(1+\sqrt{s}\big)\,\e^{C''s}\,\e^{-\beta t/s}s^{-1}
  \end{align*}
  for all $y\in M$ and $s,t>0$.
\end{corollary}

\begin{proof}
  Let $0<\beta<\alpha$. By Cauchy's inequality we obtain
  \begin{align*}
    &\int_{\rho(x,y)\geq t^{1/2}}|\Hess_xp_s(x,y)|\,\mu(dx)\\
    & \leq \lf(\int_M|\Hess_xp_s(x,y)|^2\,\e^{2\beta\,\rho^2(x,y)/s}\,\mu(dx)\r)^{1/2}\lf(\int_{\rho(x,y)\geq t^{1/2}}\e^{-2\beta\,\rho^2(x,y)/s}\,\mu(dx)\r)^{1/2}\\
    &\leq \frac{C\,\e^{C's}(1+\sqrt{s})}{s\sqrt{V(y,\!\sqrt{s})}}\sqrt{V(y,\!\sqrt{s})}\,\e^{-\beta t/s}\,\e^{\tilde{C}s}\\
    &=\frac{C(1+\sqrt{s})}{s}\,\e^{C's-\beta t/s}\,\e^{\tilde{C}s}\leq \frac{C(1+\sqrt{s})}{s}\,\e^{C''s-\beta t/s}
  \end{align*}
  where the second inequality follows from Proposition \ref{lem-Hess}
  and inequality \eqref{eqn-ele1}. This finishes the proof.
\end{proof}

\subsection{Stochastic Hessian formulas and pointwise estimates for
  Hessian of heat kernel}

In this subsection, we establish some pointwise and $L^p$-integral
estimates for the Hessian of the heat kernel when $p>2$. To this end, let
us first introduce some necessary notations.  For fixed $x\in M$, let
$B_t$ be the stochastic anti-development of $X.(x)$ which is a
Brownian motion in $T_xM$.  Let $/\!/_t\colon T_x M\to T_{X_t(x)} M$
be the parallel transport and $Q_t\colon T_{x}M\rightarrow T_{X_t}M$ be the damped parallel transport
 defined as the solution to the following pathwise ordinary covariant differential
equation along the trajectories of $X_t$, 
\begin{align}\label{eqn-Qt}
	DQ_t=-\Ric^{\sharp}Q_t\,dt,\quad Q_0={\rm id}_{T_{x}M}
\end{align}
with $DQ_t=/\!/_t\,d\,/\!/_t^{-1}Q_t$.
  For each $w\in T_{x}M$
define an operator-valued process
$W_t(\cdot,w): T_{x}M\rightarrow T_{X_t}M$ by
\begin{align*}
  W_t(\cdot, w)=Q_t\int_0^tQ_r^{-1}R(/\!/_r\,dB_r,Q_r(\cdot))Q_r(w)-Q_t\int_0^tQ_r^{-1}(\nabla \Ric^{\sharp}+d^*R)(Q_r(\cdot), Q_r(w))\,dr.
\end{align*}
This means that the process $W_t(\cdot,w)$ is the solution to the
following covariant It\^{o} equation
\begin{equation*}\left\{\begin{aligned}
      DW_t(\cdot,w)&=R(/\!/_tdB_t, Q_t(\cdot))Q_t(w)-(d^*R+\nabla
      \Ric^{\sharp})(Q_t(\cdot), Q_t(w))\,dt -\Ric^{\sharp}(W_t(\cdot,
      w))\,dt,\\ W_0(\cdot,w)&=0.
    \end{aligned}\right.\end{equation*}
First, we collect some easy curvature estimates for Riemannian
manifolds $M$ satisfying condition \eqref{eqn-H}.

\begin{lemma}\label{lem-Hc}
  Assume that \eqref{eqn-H} holds. There exist constants $C>0$,
  $\theta>0$ such that for any $t>0$,
  \begin{align}\label{R-esti}
    \sup_{x\in M}\E^x\left[\int_0^t|R|^2(X_s)\,ds+\int_0^t|\nabla \Ric^{\sharp}+d^*R|^2(X_s)\,ds\right]\leq C\e^{2 \theta t}.
  \end{align}
\end{lemma}

\begin{proof}
  By means of the trivial inequality $s\leq\e^s$ for $s\geq0$, we see
  that
  \begin{align*}
    &\sup_{x\in M}\E^x\left[\int_0^t|R|^2(X_s)\,ds+\int_0^t|\nabla \Ric^{\sharp}+d^*R|^2(X_s)\,ds\right]\\
    &\leq \sup_{x\in M}\E^x\left[\int_0^tK_1(X_s)\,ds+\int_0^tK_2(X_s)\,ds\right]\\
    &\leq \sup_{x\in M}\E^x\left[\exp\left(\int_0^tK_1(X_s)\,ds \right)\right]
      +\sup_{x\in M}\E^x\left[\exp\left(\int_0^tK_2(X_s)\,ds\right)\right].
  \end{align*}
  Recall in \cite[Lemma 3.9]{GP-15b} that for any
  $K\in \mathcal{K}(M)$, there exist positive constants $c$ and $C$
  such
  \begin{align*}
    \sup_{x\in M}\E^x\left[\exp\left(\int_0^t |K(X_s)|\,ds\right)\right]\leq C\e^{ct}.
  \end{align*}
  This immediately completes the proof by substituting $K_1$ and $K_2$
  as $K$.
\end{proof}
By this lemma and Bismut-type Hessian formula, we obtain the following
Hessian estimates for heat semigroup.
\begin{proposition}\label{lem1}
  Assume that \eqref{eqn-H} holds.  Then, for any $p\geq 2$, there
  exist positive constants $C$ and $\theta$ as in Lemma
  \textup{\ref{lem-Hc}} such that, for every $f\in \mathcal{B}_b(M)$,
  $t>0$,
  \begin{align}
    &t\,|\Hess P_tf|(x)\leq C(1+\sqrt{t})\,\e^{(2K+\theta)t}\lf(P_t|f|^p\r)^{1/p}(x); \label{lem1-i1} \\
    &t\,\big\| |\Hess P_tf|\,\big\|_p\leq C(1+\sqrt{t})\,\e^{(2K+\theta)t}\|f\|_p .\label{lem1-i2}
  \end{align}
\end{proposition}

\begin{proof}
  We start by recalling the following Bismut type formula for
  $\Hess P_t$ (see \cite{APT-03,EL-98,CLW-21}). Let $x\in D$ with
  $v,w\in T_xM$, $f\in \mathcal{B}_b(M)$ and $0<s<t$.  Suppose that
  $D_1$ and $D_2$ are regular domains such that $x\in D_1$ and
  $\bar{D}_1\subset D_2\subset D$, and denote by $\sigma,\tau$ the
  exit times of $X.(x)$ from $D_1$ and $D_2$ respectively.  Assume
  that $k,\ell$ are bounded adapted processes with paths in the
  Cameron-Martin space $L^2([0,t]; [0,1])$ such that $k_r=0$ for
  $r\geq \sigma \wedge s$, $k_0=1$, $\ell_r=1$ for
  $r\leq \sigma \wedge s$, $\ell_s=0$ for $r\geq \tau \wedge t$.
  Then, for $f\in \mathcal{B}_b(M)$, we have
  \begin{align}
    (\Hess P_tf))(v,w)
    &=-\E^x\left[f(X_t)\int_0^t\langle W_s(\dot{k}_sv,  w), /\!/_s \,d B_s\rangle\right]\notag\\
    &\quad+\E\left[f(X_t)\int_s^t\langle Q_r(\dot{\ell}_rw), /\!/_r\,  d B_r\rangle
      \int_0^{s}\langle Q_r(\dot{k}_rv), /\!/_r\,d B_r\rangle\right].\label{HessBismutFormula}
  \end{align}
  Note that under condition $\bf (H)$ one has $|Q_t|^2\leq \e^{2Kt}$
  and according to the It\^{o} formula, then for $t<\tau<\tau_D$ and
  $0<\delta<2K$,
  \begin{align*}
    d |W_s(v, w)|^2&=2\,\langle R(/\!/_s d B_s, Q_s(v))Q_s(w), W_s(v,w)\rangle+2 |R^{\sharp,\sharp}(Q_s(v),Q_s(w))|_{\rm HS}^2\, d s\\
                   &\quad- 2\,\langle (d^* R+\nabla \Ric^{\sharp})(Q_s(v), Q_s(w)), W_s(v,w) \rangle\, d s\\
                   &\quad- 2\Ric(W_s(v,w),W_s(v,w))\,d s\\
                   &\leq 2\,\langle R(/\!/_s d B_s, Q_s(v))Q_s(w), W_s(v,w)\rangle+ 2\e^{4Ks}K_1(X_s)\, d s\\
                   &\quad+ 2\sqrt{K_2(X_s)} \,\e^{2Ks} |W_s(v,w)|\, d s+2K |W_s(v,w)|^2\, d s\\
                   &\leq 2\,\langle R(/\!/_s d B_s, Q_s(v))Q_s(w), W_s(v,w)\rangle+ 2\e^{4Ks}K_1(X_s)\, d s\\
                   &\quad+ (2\delta)^{-1}\e^{4Ks}K_2(X_s) \, d s+(2K+\delta) |W_s(v,w)|^2\, d s.
  \end{align*}
  From this and Lemma \ref{lem-Hc} there exist constants $C>0$ and
  $\theta>0$ such that
  \begin{align}\label{est-Wt}
    \E|W_{t\wedge \tau_D}(v,w)|^2
    &\leq 2\e^{(2K+\delta)t}\left(\E\left[\int_0^t\e^{(2K-\delta)s}K_1(X_s)\,ds\right]+(4\delta)^{-1}\E\left[\int_0^t\e^{(2K-\delta)s}K_2(X_s)\,ds\right]\right) \\
    &\leq  2\max\{1,(4\delta)^{-1}\}\e^{4Kt} \E\left[\int_0^tK_1(X_s)\,ds+\int_0^tK_2(X_s)\,ds\right] \notag\\
    &\leq  C \,\e^{(4K+2\theta)t}.\notag
  \end{align}
  Since the upper of $ \E|W_{t\wedge \tau_D}(v,w)|^2$ is independent
  of the domain $D$, taking a sequence of domains $D_n$ increasing to
  $M$, we conclude that
  $ \E|W_{t}(v,w)|^2\leq C \,\e^{(4K+2\theta)t}.$ All these estimates
  allow to verify that formula \eqref{HessBismutFormula} holds for the
  deterministic functions $k_s=\frac{t-2s}{t}\vee 0$ and
  $\ell_s=1\wedge \frac{2(t-s)}{t}$ as well.  Then, by estimate
  \eqref{est-Wt} and $|Q_s|\leq \e^{Ks}$, we get for any $p\geq 2$,
  \begin{align}\label{Hess-pointwise-esti}
    |\Hess P_tf|
    &\leq  \frac{2}{t}(P_t|f|^p)^{1/p}\left[\int_0^{t/2}\E|W_s(v,w)|^2\,ds\right]^{1/2}\notag\\
    &\quad+ \frac{4}{t^2}(P_t|f|^p)^{1/p}\E\left[\Big(\int_0^{t/2}\langle Q_s(w),/\!/_sdB_s\rangle\Big)^2
      \E\Big[\Big(\int_{t/2}^t\langle Q_s(v), /\!/_sdB_s \rangle\Big)^2\Big|\SF_{t/2}\Big]\right]^{1/2}\notag\\
    &\leq \frac{2}{t}(P_t|f|^p)^{1/p}\left(\int_0^{t/2}\e^{(4K+2\theta)s}\,ds\right)^{1/2}\notag\\
    &\quad+ \frac{4}{t^2}(P_t|f|^p)^{1/p}\E\def\ovz{\overrightarrow}
      \left[\int_0^{t/2}\e^{2Ks}\,ds\right]^{1/2}\left(\int_{t/2}^t\e^{2Ks}\,ds\right)^{1/2}\notag\\
    &\leq  \frac{C}{\sqrt{t}}\,\e^{(2K+\theta)t}(P_t|f|^p)^{1/p}+
      \frac{C}{t}\,\e^{2Kt}(P_t|f|^p)^{1/p}\notag\\
    &\leq  \frac{C(1+\sqrt{t})}{t}\,\e^{(2K+\theta)t}(P_t|f|^p)^{1/p}
  \end{align}
  which proves \eqref{lem1-i1}.  Moreover, integration with respect to
  the measure $\mu$ yields
  \begin{align*}
    \mu\big(|\Hess P_tf|^p\big)^{1/p}\leq  \frac{C(1+\sqrt{t})}{t}\,\e^{(2K+\theta)t} \mu(|f|^p)^{1/p}.
  \end{align*}
  This proves \eqref{lem1-i2} and finishes the proof.
\end{proof}

The following proposition gives a pointwise estimate for the Hessian
of heat semigroup.

\begin{proposition}\label{lem5}
  Assume that \eqref{eqn-H} holds.  Then there exist constants
  $C, \theta>0$ such that
  \begin{align}\label{hess-gra}
    |\Hess P_tf|\leq
    \e^{2Kt}\left(P_t|\Hess f|^2\right)^{1/2}+C\e^{(2K+\theta) t}\,\left(P_t|df|^2\right)^{1/2}
  \end{align}
  for $t>0$ and $f\in C_0^{\infty}(M)$.
\end{proposition}

\begin{proof}
  Recall the following relation
  \begin{align}\label{eqn-Re1}
    d\Delta f=\tr \Hess(df)-df(\Ric^{\sharp})
  \end{align}
  where $\Hess(df)=\nabla^2df$.
  Moreover, using the differential Bianchi identity
  \cite[p.~185]{CH97}, we have
  \begin{align}\label{eqn-Re2}
    \Hess(\Delta f)=\tr \Hess(\Hess f)-2(\Hess f)
    (\Ric^{\sharp}\otimes {\rm id}-R^{\#,\#})-df(d^*R+\nabla \Ric^{\sharp}).
  \end{align}
  Now, for $s>0$ and $v,w\in T_xM$, let
  \begin{align*}
    N_s(v,w):=(\Hess P_{t-s}f)(Q_s(v), Q_s(w))(X_s)+(dP_{t-s}f)(W_s(v,w))(X_s).
  \end{align*}
  We claim that $N_s$ is a local martingale.  Combining
  \eqref{eqn-Re1}, \eqref{eqn-Re2} and applying It\^o's formula, we
  obtain
  \begin{align*}
    dN_s(v,w)
    &=\big(\nabla_{/\!/_sdB_s}\Hess P_{t-s}f\big)(Q_s(v), Q_s(w))+(\Hess P_{t-s}f)(\frac{D}{ds}Q_s(v),Q_s(w))\,ds\\
    &\quad+(\Hess P_{t-s}f)(Q_s(v),\frac{D}{ds}Q_s(w))\,ds+\partial_s(\Hess P_{t-s}f)(Q_s(v),Q_s(w))\,ds
    \\
    &\quad+\tr\Hess(\Hess P_{t-s}f)(Q_s(v), Q_s(w))\,ds+(\nabla_{/\!/_sdB_s}d P_{t-s}f)(W_s(v,w))\\
    &\quad+(dP_{t-s}f)(DW_s(v,w))+\langle d(dP_{t-s}f), DW_s(v,w) \rangle+\partial_s(dP_{t-s}f)(W_s(v,w))\,ds\\
    &\quad+\tr\Hess (dP_{t-s}f)(W_s(v,w))\,ds\mequal 0,
  \end{align*}
  where $\mequal$ denotes equality modulo the differential of a local
  martingale, and here we used for the quadratic covariation the
  formula
  \begin{align*} [d(dP_{t-s}f), DW_s(v,w)]_s=(\Hess
    P_{t-s}f)(R^{\#,\#}(Q_s(v),Q_s(w)))\,ds.
  \end{align*}
  Therefore, $N_s$ is a local martingale. Recall that condition
  $\bf (H)$ implies
  \begin{align*}
    |Q_t|^2\leq \e^{2Kt} \quad\text{and}\quad \E|W_t(v,w)|^2\leq C \e^{(4K+2\theta)t}  \end{align*}
  for some positive constants. In view of estimate
  \eqref{Hess-pointwise-esti}, $|\Hess P_{t-s}f|$ is bounded on
  $[0,t-\varepsilon]\times M$, and
  $$|\nabla P_{t-s}f|\leq \e^{K(t-s)}\,\|\,|\nabla f|\,\|_{\infty}$$
  as a consequence of the derivative formula (see e.g. \cite[Theorem 2.2.3]{Wbook14}):
  for all $v\in T_xM$ and $x\in M$,
\begin{align}\label{derivative-formula}
	\langle (\nabla P_tf)(x),v\rangle=\E^x\left[\langle(\nabla f)(X_t),Q_tv\rangle\right].
\end{align} 
Thus $N_s$ is a martingale on the time
  interval $[0,t-\varepsilon]$, and by taking expectation at time $0$
  and $t-\varepsilon$, we arrive at
  \begin{align}\label{Hessian-formula}
    (\Hess P_tf)(v,w)=\E\Big[(\Hess P_{\varepsilon} f)(Q_{t-\varepsilon}(v),Q_{t-\varepsilon}(w))(X_t)+ dP_{\varepsilon}f(W_{t-\varepsilon}(v,w))(X_{t-\varepsilon})\Big].
  \end{align}
  Letting $\varepsilon$ tend to $0$, it then follows that
  \begin{align*}
    |\Hess P_tf(v,w)|
    & \leq \Big|\E\Big[\Hess f(Q_t(v), Q_t(w))(X_t)\Big]\Big|+ \Big|\E\Big[df(W_t(v,w))(X_t)\Big]\Big|\\
    &\le P_t\Big[|\Hess f|\,|Q_t(v)|\,|Q_t(w)|\Big]+(P_t|df|^2)^{1/2}\E\Big[|W_t(v,w)|^2\Big]^{1/2}\\
    &\le (P_t|\Hess f|^2)^{1/2}\,\e^{2Kt}+C\,\e^{(2K+\theta)t}\,(P_t|df|^2)^{1/2}.\qedhere
  \end{align*}
\end{proof}

\begin{remark}
  Let
  $$\|R\|_{\infty}:=\sup_{x\in M} |R|(x)\quad \mbox{ and} \quad
  \|\nabla \Ric^{\sharp}+d^*R\|_{\infty}:=\sup_{x\in M} |\nabla
  \Ric^{\sharp}+d^*R|(x).$$ If assumption \eqref{eqn-H} is replaced by
  \eqref{eqn-Ric} with $K\geq 0$, $\|R\|_{\infty}<\infty$ and
  $\|\nabla \Ric^{\sharp}+d^*R\|_{\infty}<\infty$, then it is
  straightforward to see from \eqref{est-Wt} that for $0<\delta<2K$,
  \begin{align*}
    \E|W_t(v,w)|^2 \leq  \e^{(2K+\delta)t} \E\left[\int_0^t\e^{(2K-\delta)s}\|R\|_{\infty}^2\,ds+\int_0^t\e^{(2K-\delta)s}\|\nabla \Ric^{\sharp}+d^*R\|_{\infty}^2\,ds\right]\leq C\e^{4Kt} t
  \end{align*}
  for some explicit constant $C>0$. Proceeding as in the proof of
  Proposition \ref{lem5} above, we obtain in this case
  \begin{align*}
    |\Hess P_tf|\leq  \e^{2Kt}\big(P_t|\Hess f|^2\big)^{1/2}+C\sqrt{t}\e^{2K t}\big(P_t|df|^2\big)^{1/2}.
  \end{align*}
\end{remark}

Proposition \ref{lem5} is used to establish pointwise Gaussian upper
bounds for the Hessian of the heat kernel via inequality
\eqref{hess-gra}. Note that Carron \cite{Carron} applied pointwise
Gaussian upper bounds for the gradient of the heat kernel to the Riesz
transform on complete manifolds whose Ricci curvature satisfies a
quadratic decay control. Here, we shall use the following Hessian heat
kernel estimate in the next section to establish the
Calder\'{o}n-Zygmund inequality.

\begin{proposition}\label{lem-hess-heat-kernel}
  Assume that \eqref{eqn-H} holds. Fix $\az\in (0,1/4)$ as in Lemma
  \textup{\ref{lem-Gausspe}} and let $\theta>0$ be as in
  \eqref{R-esti}.  Then there exist $0<\beta<2\alpha$ and $C_3>0$ such
  that
  \begin{align*}
    |\Hess_xp_{t}(x,y)|&\leq \frac{C(1+\sqrt{t})}{tV(y,\!\sqrt{t})}\exp\left(-\beta\,\frac{\rho^2(x,y)}{t}+\frac{1}{2}\big(C_{3}+\theta\big)t\right).
  \end{align*}
  for any $t>0$ and $x$, $y\in M$.
\end{proposition}

\begin{proof}
  For $t>0$ and $x$, $y\in M$, write
  \begin{align*}
    \Hess_xp_{2t}(x,y)=\Hess P_t(p_t(\cdot,y))(x)
  \end{align*}
  Applying Proposition \ref{lem5} and taking $\gz\in (0,2\az)$, we
  have
  \begin{align*}
    \big|\Hess P_t(p_t(\cdot,y))(x)\big|
    &\leq  \e^{2Kt}\left(\int |\Hess_zp_t(z,y)|^2\,\e^{\gamma{\rho^2(z,y)}/t}
      \,\e^{-\gamma{\rho^2(z,y)}/t}\,|p_t(x,z)|\,\mu(dz)\right)^{1/2} \\
    &\quad+C\e^{(2K+\theta)t}\left(\int |{\nabla}_z p_t(z,y) |^2\,\e^{\gamma{\rho^2(z,y)}/t}\,\e^{-\gamma{\rho^2(z,y)}/t}\,|p_t(x,z)|\,\mu(dz)\right)^{1/2}. \notag 
  \end{align*}
  Moreover, by Lemma \ref{lem4} and Proposition \ref{lem-Hess}, we
  know that there exists $C'>0$ such that
  \begin{align*}
    \int_M \lf(|{\nabla}_zp_t(z,y)|^2\,\e^{\gamma{\rho^2(z,y)}/t}+ |\Hess_zp_t(z,y)|^2
    \,\e^{\gamma{\rho^2(z,y)}/t}\r)\,\mu(dz)\leq \frac{C(1+Kt)}{t^2V(y,\!\sqrt{t})}\,\e^{2C't}.
  \end{align*}
  This further implies that there exists a constant $C_{3}>0$ such
  that
  \begin{align*}
    \big|\Hess P_t(p_t(\cdot,y))(x)\big|
    &\leq 
      \frac{C(1+\sqrt{t})}{t\sqrt{V(y,\!\sqrt{t})}}\,\e^{((C'+2K)+\theta)t}\sup_{z\in M}
      \left\{\e^{-\gamma{\rho^2(z,y)}/{t}}|p_t(x,z)|\right\}^{1/2}\\
    &\le \frac{C(1+\sqrt{t})}{{t}{V(y,\!\sqrt{t})}}\,\e^{(C_{3}+\theta)t}
      \,\e^{-\beta {\rho^2(x,y)}/t}
  \end{align*}
  for $\beta$ small enough, where in the last inequality we have used
  the estimate
  \begin{align*}
    \sup_{z\in M}\Big\{\e^{-\gamma{\rho^2(z,y)}/t}|p_t(x,z)|\Big\}\leq \frac{\e^{C_1Kt}}{V(y,\!\sqrt{t})}\e^{-2\beta {\rho^2(x,y)}/t},
  \end{align*}
  which can be deduced from Lemma \ref{lem-Gausspe}. Thus, we conclude
  \begin{align*}
    |\Hess_xp_{2t}(x,y)|\leq \frac{C(1+\sqrt{t})}{tV(y,\!\sqrt{t})}\,\e^{(C_{3}+\theta)t}\,\e^{-\beta{\rho^2(x,y)}/t}
  \end{align*}
  which finishes the proof of Proposition \ref{lem-hess-heat-kernel}.
\end{proof}

By interpolating the $L^p$-Hessian inequality with the $L^2$-Gaffney
off-diagonal estimates, one obtains $L^p$-Gaffney off-diagonal
estimates for any $p>2$ as follows.

\begin{proposition}\label{Lp-Gaffney-estimate}
  Assume that \eqref{eqn-H} hold. Then for $p>2$ there exists
  constants $C,C_4>0$ such that for all $t\in (0,\fz)$, all Borel
  subsets $E,F\subset M$ with compact closure, and all $f\in L^p(M)$
  with {\rm supp}$f\subset E$,
  \begin{align*}
    \big\|\mathbbm{1}_Ft |\Hess P_tf|\big\|_p\leq 
    C(1+\sqrt{t})\,\e^{(2K+\theta)t}\,
    \e^{-{C_4\rho^2(E,F)}/t}\,\|f\|_p.
  \end{align*}
\end{proposition}

\begin{proof}
  Let $p>2$ and $t>0$. For $E,F$ and $f$ as above, by inequality
  \eqref{lem1-i1}, there exist positive constants $C$ and $C_2$ such
  that
  \begin{align*}
    \int_Ft^p|\Hess P_tf|^p(x)\,\mu(dx)
    &\leq C\e^{(2K+\theta)pt}(1+\sqrt{t})^p\int_F(P_t|f|^{p/2})^2(x)\,\mu(dx)\\
    &\leq C\e^{-{2C_2 \rho^2(E,F)}/t}\,\e^{(2K+\theta)pt}\,(1+\sqrt{t})^p \int_{E} |f|^p\,\mu(dx)
  \end{align*}
  where the last inequality follows from the $L^2$-Gaffney
  off-diagonal estimates for $P_tf$ (Lemma \ref{lem-Gaffney}),
  i.e.~the existence of positive constants $C$ and $C_2$ such that
  \begin{align*}
    \big\|\mathbbm{1}_F |P_tf^{p/2}|\,\big\|_{2}\leq C \,\e^{-{C_2 \rho^2(E,F)}/t}\left\|\mathbbm{1}_E f^{p/2}\right\|_{2}
  \end{align*}
  for all $f\in L^p(M)$ with $\text{supp} f\subset E$.
\end{proof}

\section{Calder\'on-Zygmund inequality for $p\neq 2$}\label{s3}
In this section, we prove the main results of this paper. We first
prove Theorem \ref{them-1} in Section \ref{s3.1}, namely {\bf CZ($p$)}
for $p\in (1,2)$ under the condition \eqref{eqn-Ric}; then in Section
\ref{s3.2}, we prove Theorem \ref{theorem-2}, namely, {\bf CZ($p$)}
for $p>2$ under the condition \eqref{eqn-H}. Recall that {\bf CZ($2$)}
always holds under condition \eqref{eqn-Ric}.

\subsection{The case $p\in (1,2)$}\label{s3.1}

Let $M$ be a complete Riemannian manifold satisfying \eqref{eqn-Ric}.
In this subsection, we prove Theorem \ref{them-1} and show that the
Calder\'on-Zygmund inequality holds for $p\in (1,2)$. To be precise,
let
\begin{align}\label{eqn-defT}
  T=\Hess\,(\Delta+\sz)^{-1}=\dint_0^\fz \e^{-\sz t} \Hess P_t\,dt.
\end{align}
We show that $T$ is bounded on $L^p(M)$ for any $p\in (1,2)$, provided
$\sigma>0$ is large enough. Since $T$ is already bounded on $L^2(M)$,
using interpolation, it suffices to prove that $T$ is of weak type
$(1,1)$, that is for some constant $C$,
\begin{align}\label{eqn-bdT}
  \mu\big(\{x\in M\colon |Tf(x)|>\lambda\}\big)\leq C\dfrac{\|f\|_1}{\lambda}
\end{align}
for all $\lambda>0$ and $f\in L^1(M)$. To this end, we need the
following technical lemma from \cite[Section 4]{PTTS-2004} on the
finite overlap property of $M$.

\begin{lemma}\textup{\cite{PTTS-2004}}\label{lem-fop}
  Assume that \eqref{eqn-GD} holds. There exists a countable subset
  $\mathcal{C}=\{x_j\}_{j\in \Lambda}\subset M$ such that
  \begin{itemize}
  \item [{\rm (i)}] $M=\cup_{j\in \Lambda} B(x_j,1)$;

  \item [{\rm (ii)}] $\{B(x_j,1/2)\}_{j\in\Lambda}$ are disjoint;

  \item [{\rm (iii)}] there exists $N_0\in\N$ such that for any
    $x\in M$, there are at most $N_0$ balls $B(x_j,4)$ containing $x$;

  \item [{\rm (iv)}] for any $c_0\geq 1$, there exists $C>0$ such that
    for any $j\in\Lambda$, $x\in B(x_j,c_0)$ and $r\in (0,\fz)$,
    \begin{align*}
      \mu\lf(B(x,2r)\cap B(x_j,c_0)\r)\le C \mu\lf(B(x,r)\cap B(x_j,c_0)\r) 
    \end{align*}
    and
    \begin{align*}
      \mu(B(x,r))\le C \mu\lf(B(x,r)\cap B(x_j,c_0)\r) 
    \end{align*}
    for any $x\in B(x_j,c_0)$ and $r\in (0,2c_0]$.

  \end{itemize}
\end{lemma}

The following lemma provides the localization argument to prove
\eqref{eqn-bdT}.

\begin{lemma}\label{lem-bcT}
  Assume that \eqref{eqn-Ric} holds.  Let
  $\mathcal{C}=\{x_j\}_{j\in\Lambda}$ be a countable subset of $M$
  having finite overlap property as in Lemma
  \textup{\ref{lem-fop}}. Let $\sigma>C''$ where $C''$ is as in
  Corollary \ref{prop1}. Suppose that there exists a positive constant
  $C$ such that
  \begin{align}\label{eqn-localT}
    \mu\lf(\{x\colon  \mathbbm{1}_{B(x_j,2)}|Tf(x)|>\lambda\}\r)\leq \frac{C}{\lambda}\|f\|_1 
  \end{align}
  for any $j\in \Lambda$, $\lz\in (0,\fz)$ and
  $f\in C_0^\fz (B(x_j,1))$. Then \eqref{eqn-bdT} holds for any
  $f\in C_0^\fz(M)$.
\end{lemma}

\begin{proof}
  For $j\in \Lambda$, let $B_j:=B(x_j,1)$ and let
  $\{\varphi_j\}_{j\in\Lambda}$ be a $C^{\infty}$-partition of unity
  such that $0\le \varphi_j\le 1$ and is supported in $B_j$.  Then,
  for any $f\in C_0^{\infty}(M)$ and $x\in M$, write
  \begin{align*}
    Tf(x)= \sum_{j\in\Lambda} \mathbbm{1}_{2B_j}T(f\varphi_j)(x)+\sum_{j\in\Lambda}(1-\mathbbm{1}_{2B_j})T(f\varphi_{j})(x),
  \end{align*}
  which yields that for any $\lambda>0$,
  \begin{align*}
    \mu(\{x\colon |Tf(x)|>\lambda\})
    &\leq  \mu\bigg(\bigg\{x\colon \dsum_{j\in\Lambda} \mathbbm{1}_{2B_j}|T(f\varphi_j)(x)|>\frac{\lambda}{2}\bigg\}\bigg)
      + \mu\bigg(\bigg\{x\colon \dsum_{j\in\Lambda}(1-\mathbbm{1}_{2B_j})|T(f\varphi_j)(x)|>\frac{\lambda}{2}\bigg\}\bigg)\\
    &=:\mathrm{I}_1+\mathrm{I}_2.
  \end{align*}
  For $\mathrm{I}_1$, by Lemma \ref{lem-fop}(iii) and
  \eqref{eqn-localT}, we have
  \begin{align}\label{lem-eI1}
    \mathrm{I}_1\le  \dsum_{j\in\Lambda}\mu\bigg(\bigg\{x:  \mathbbm{1}_{2B_j}|T(f\varphi_j)(x)|>\frac{\lambda}{2N_0}\bigg\}\bigg)
    \ls\frac{1}{\lambda}\|f\|_1
  \end{align}
  as desired, where the notation $a\ls b$ means $a\leq Cb$ for some
  constant $C$.

  To bound $\mathrm{I}_2$, again by Lemma \ref{lem-fop}(iii), since
  $\varphi_j$ is supported in $B_j$, it is easy to see that
  \begin{align*}
    \dsum_{j\in\Lambda}|(1-\mathbbm{1}_{2B_j})(x)\varphi_j(y)|\leq N_0 \mathbbm{1}_{\{\rho(x,y)\geq 1\}}.
  \end{align*}
  Hence, according to the definition of $T$ in \eqref{eqn-defT} and
  Corollary \ref{prop1}, we get
  \begin{align*}
    \mathrm{I}_2
    &\leq \frac{2}{\lz} \dsum_{j\in\Lambda} \lf\|\,\big|\lf(1-\mathbbm{1}_{2B_j}\r)
      T\lf(f\varphi_j\r)\big|\,\r\|_{1}
    \\
    &\ls \frac{1}{\lz} \int_M\lf(\int_0^{\infty}\e^{-\sigma t} \int_M \lf|\Hess_xp_t(x,y)\r|\sum_{j\in\Lambda}|(1-\mathbbm{1}_{2B_j})(x)\varphi_j(y)||f(y)|\,\mu(dy)\,dt\r)\mu(dx)\\
    &\ls\frac{1}{\lz}\int_{M}\int_0^{\infty}\e^{-\sigma t} \lf(\int_{\rho(x,y)\geq 1} |\Hess_xp_t(x,y)|\,\mu(dx)\r)\,dt\, |f(y)|\,\mu(dy)\\
    &\leq \frac{1}{\lz} \int_M |f(y)|\,\mu(dy) \int_0^{\infty}\e^{-\sigma t}(1+\sqrt{t})\,\e^{C''t}\,\e^{-\beta/ t}t^{-1}\,dt,
  \end{align*}
  where $\bz\in (0,\az)$. Thus, since $\sigma >C''$, we obtain
  \begin{align*}
    \mathrm{I}_2\ls \frac{1}{\lz}\int_0^{\infty}\e^{(C''-\sigma )t
    -{\beta}/{t}}\frac{(1+\sqrt{t})}{t}\,dt \,\|f\|_1\ls \frac{1}{\lz}\|f\|_1,
  \end{align*}
  which combined with the estimate about $I_1$ in \eqref{lem-eI1}
  finishes the proof of Lemma \ref{lem-bcT}.
\end{proof}

To prove property \eqref{eqn-localT}, we remove the subscript $j$ and
simply write $B$ for each $B_j:=B(x_j,1)$.  Let $c_0\geq 1$. By Lemma
\ref{lem-fop}(iv), we have that $(c_0B,\mu,\rho)$ is a metric measure
subspace satisfying the {\it volume doubling property} that there
exists $C_D\ge 1$ such that
\begin{align}\label{eqn-D}
  \mu\lf(B(x,2r)\cap c_0B\r)\leq C_D \,\mu\lf(B(x,r)\cap c_0B\r) \tag{\bf D}
\end{align}
for all $x\in c_0B$ and $r>0$.

We also need the following Calder\'on-Zygmund decomposition from
\cite{CW-book}.

\begin{lemma}[\cite{CW-book}]\label{lem-czdecom}
  Let $(\mathcal{X},\nu,\rho)$ be a metric measure space satisfying
  \eqref{eqn-D} with $c_0B$ replaced by $\mathcal{X}$.  Let
  $f\in L^1(\mathcal{X})$ and $\lz\in (0,\fz)$. Assume
  $\|f\|_{L^1}< \lz \nu(\mathcal{X})$.  Then $f$ has a decomposition
  of the form $$\textstyle f=g+b=g+\sum_ib_i$$ such that
  \begin{enumerate}[{\rm(a)}]
  \item $g(x)\leq C\lambda$ for almost all $x\in M$;
  \item there exists a sequence of balls $B_i=B(x_i,r_i)$ so that the
    support of each $b_i$ is contained in $B_i$:
    \begin{align*}
      \int_{\mathcal{X}} |b_i(x)|\, \nu(dx)\leq C\lambda \nu(B_i)\quad \text{and}\quad
      \int_{\mathcal{X}} b_i(x)\, \nu(dx)=0;
    \end{align*}
  \item
    $\displaystyle\sum_i \nu(B_i)\leq
    \frac{C}{\lambda}\int_{\mathcal{X}} |f(x)|\, \nu(dx)$;
  \item there exists $k_0\in \mathbb{N}^*$ such that each point of $M$
    is contained in at most $k_0$ balls $B_i$.
  \end{enumerate}
\end{lemma}

\begin{lemma}\label{lem-le}
  Assume that \eqref{eqn-Ric} holds.  Let $\lz\in (0,\fz)$ and
  $f\in L^1(B)$ be as in Lemma \textup{\ref{lem-czdecom}}. Assume that
  $\{b_i\}$ is the sequence of bad functions as in Lemma
  \textup{\ref{lem-czdecom}} and $\{P_{t}^{\sz}\}_{t\ge 0}$ the heat
  semigroup associated to $-(\Delta+\sz)$ with $\sz>0$. Then there
  exist $\sigma >0$ and $C>0$ independent of $f$ such that
  \begin{align*}
    \Big\|\sum_iP^{\sigma}_{t_i}b_i\Big\|_2^2\leq C\lambda \|f\|_1
  \end{align*}
  where $t_i=r_i^2$ with $r_i$ denoting the radius of the ball $B_i$
  as in Lemma \ref{lem-czdecom}(b).
\end{lemma}

\begin{proof}
  Recall that $\supp b_i \subset B(x_i, \sqrt{t_i})$. Using the upper
  bound of the heat kernel in Lemma \ref{heat-kernel-upper-bound} and
  Lemma \ref{lem-czdecom}, we have for $x\in M$,
  \begin{align*}
    |P_{t_i}^{\sz}b_i(x)|
    &\leq \int_M \frac{\e^{-\sz' t_i-\alpha \frac{\rho^2(x,y)}{t_i}}}{V(x,\sqrt{t_i})}|b_i(y)|\,\mu(dy)\\
    & \leq  \frac{C}{V(x,\!\sqrt{t_i})} \,\e^{-\sz' t_i-\alpha'\frac{\rho^2(x,x_i)}{t_i}}\int_{B_i} |b_i(y)|\,
      \mu(dy)\\
    &\leq  C_2 \lambda \int_M \frac{\e^{-\sz' t_i-\alpha{''}\frac{{ \rho^2(x,y)}}{t_i}}}{V(x,\sqrt{t_i})}\mathbbm{1}_{B_i}(y)\,\mu(dy),
  \end{align*}
  where $\sz'=\sz-C_1K>0$ and $0<\az''<\az'<\az$.  It is therefore
  sufficient to verify that
  \begin{align}\label{lem-le1}
    \biggl\|\sum_i \int_M \frac{\e^{-\sz' t_i-\alpha'' \frac{\rho^2(\cdot, y)}{t_i}}}
    {V(\cdot, \sqrt{t_i})}\mathbbm{1}_{B_i}(y)\,\mu(dy)\biggr\|_2\ls\biggl\|\sum_i\mathbbm{1}_{B_i}\biggr\|_2,
  \end{align}
  since from this and Lemma \ref{lem-czdecom} we obtain as
  consequence, as desired,
  \begin{align*}
    \Big\|\sum_iP_{t_i}^{\sigma}b_i\Big\|^2_2\ls\lambda^2\biggl\|\sum_i\mathbbm{1}_{B_i}\biggr\|^2_2\ls \lambda^2 \sum_i \mu(B_i)\ls \lambda \|f\|_1.
  \end{align*}
  In order to prove \eqref{lem-le1}, we write by duality
  \begin{align}\label{lem-le2}
    &\biggl\|\sum_i \int_M \frac{\e^{-\sz' t_i-\alpha''\frac{\rho^2(\cdot,y)}{t_i}}}{V(\cdot,\!\sqrt{t_i})}\mathbbm{1}_{B_i}(y)\,\mu(dy)\biggr\|_2\\ \notag
    &\qquad =\sup_{\|u\|_2=1}\lf|\int_M\lf(\sum_i\int_M\frac{\e^{-\sz' t_i-\alpha'' \frac{\rho^2(x,y)}{t_i}}}{V(x,\sqrt{t_i})}\mathbbm{1}_{B_i}(y)\,\mu(dy)\r)u(x)\,\mu(dx)\r|\\ \notag
    &\qquad
      \leq \sup_{\|u\|_2=1}\int_M \sum_i\lf(\int_M \frac{\e^{-\sz' t_i-\alpha''\frac{\rho^2(x,y)}{t_i}}}{V(x,\sqrt{t_i})}|u(x)|\,\mu(dx)\r)\mathbbm{1}_{B_i}(y)\,\mu(dy).
  \end{align}
  By \eqref{eqn-gdxy}, we have for any $x\in M$ and $y\in B_i$,
  \begin{align*}
    V(y,\!\sqrt{t_i})\leq C\lf(1+\frac{\rho(x,y)}{\sqrt{t_i}}\r)^d \e^{C(t_i^{1/2}+\rho(x,y))}
    V(x,\!\sqrt{t_i})
  \end{align*}
  with $\dz\in [0,2)$. From this, we obtain that there exist
  $0<\tilde{\az}<\az'''<\az''$ such that
  \begin{align*}
    &\int_M \frac{\e^{-\sz' t_i-\alpha'' \frac{\rho^2(x,y)}{t_i}}}{V(x,\!\sqrt{t_i})}\,|u(x)|\,\mu(dx)\\
    &\hs\ls\frac{\e^{-\frac{1}{2}\sz' t_i}}{V(y,\!\sqrt{t_i})}\int_M \e^{-\alpha'''\frac{\rho^2(x,y)}{t_i}}|u(x)|\,\mu(dx)\\
    &\hs\ls\frac{1}{V(y,\!\sqrt{t_i})}\Bigg(\int_{\rho(x,y)<\sqrt{t_i}}|u(x)|\,\mu(dx)
      +\sum_{k=0}^\fz\int_{2^k\!\sqrt{t_i}\leq \rho(x,y)< 2^{k+1}\!\sqrt{t_i}}\e^{-\alpha'''\frac{\rho^2(x,y)}{t_i}}|u(x)|\,\mu(dx)\Bigg)\\
    &\hs\leq \frac{1}{V(y,\!\sqrt{t_i})}\lf(\int_{B(y,\!\sqrt{t_i})}|u(x)|\,\mu(dx)+\sum_{k=0}^\fz\e^{-\alpha''' 2^{2k}}\int_{B(y,2^{k+1}\sqrt{t_i})}|u(x)|\,\mu(dx)\r)\\
    &\hs=\lf(1+\sum_{k=0}^\fz\frac{V(y,\!2^{k+1}\sqrt{t_i})}{V(y,\!\sqrt{t_i})}\,\e^{-\alpha'''2^{2k}}\r)(\SM u)(y)\\
    &\hs\leq \left(1+C\sum_{k=0}^\fz2^{(k+1)d}\e^{C2^{k+1}\sqrt{t_i}}\e^{-\tilde{\az} 2^{2k}}\right)(\SM u)(y)\\
    & \hs\leq \left(1+C\sum_{k=0}^\fz2^{(k+1)d}\e^{C2^{k+1}}\e^{-\tilde{\az} 2^{2k}}\right)(\SM u)(y)\ls (\SM u)(y),
  \end{align*}
  where \begin{align*} (\SM
    u)(y):=\sup_{r>0}\frac{1}{V(y,r)}\int_{B(y,r)}|u(x)|\,\mu(dx)
        \end{align*}
        denotes the Hardy-Littlewood maximal function of $u$.  This
        together with \eqref{lem-le2} and the $L^2$-boundedness of
        $\SM $ gives
        \begin{align*}
          \biggl\|\sum_i\int_M \frac{\e^{-\alpha''\frac{d^2(\cdot, y)}{t_i}}}{V(\cdot,\!\sqrt{t_i})}\mathbbm{1}_{B_i}(y)\,\mu(dy)\biggr\|_2\ls
          \sup_{\|u\|_2=1}\int_M (\SM u)(y)\sum_{i}\mathbbm{1}_{B_i}(y)\,\mu(dy)
          \ls\biggl\|\sum_i\mathbbm{1}_{B_i}\biggr\|_2,
        \end{align*}
        which shows that \eqref{lem-le1} holds true and finishes the
        proof of Lemma \ref{lem-le}.
      \end{proof}


      With the help of Lemmas \ref{lem-bcT} through \ref{lem-le}, we
      are now in position to the proof of Theorem \ref{them-1}.

\begin{proof}[Proof of Theorem \ref{them-1}] 
  Recall that $T=\Hess(\Delta+\sz)^{-1}$. We choose $\sigma$ big
  enough such that $\sigma>\max\{C'', C_1K\}$, where $\tilde{\alpha}$
  is as in proof of Lemma \ref{lem-le}.  By Lemma \ref{lem-bcT}, it
  suffices to prove
  \begin{align}\label{eqn-localT-1}
    \mu\big(\{x\in 2B\colon  |Tf(x)|>\lambda\}\big)\ls \dfrac{\|f\|_1 }{\lambda}
  \end{align}
  for all $f\in C_0^\fz(B)$. By means of Lemma \ref{lem-czdecom} with
  $\mathcal{X}=B$, we deduce that $f$ has a decomposition
  $$\textstyle f=g+b=g+\sum_ib_i$$ which implies
  \begin{align}\label{eqn-them-8}
    \mu\big(\{x\in 2B: |Tf(x)|>\lambda\}\big)
    &\leq \mu\lf(\lf\{x\in 2B:|Tg(x)|>\frac{\lambda}{2}\r\}\r)+
      \mu\lf(\lf\{x\in 2B: |Tb(x)|>\frac{\lambda}{2}\r\}\r)\\ \notag
    &=:\mathrm{I}_1+\mathrm{I}_2.
  \end{align}
  Using the facts that $T$ is bounded on $L^2(M)$ and that
  $|g(x)|\leq C\lambda$, we obtain as desired
  \begin{align}\label{eqn-them-9}
    \mathrm{I}_1\ls\lambda^{-2}\|Tg\|_2^2\ls\lambda^{-2}\|g\|_2^2
    \ls\lambda^{-1}\|g\|_1\ls\lambda^{-1}\|f\|_1.
  \end{align}

  We now turn to the estimate of $\mathrm{I}_2$. Recall that
  $\{P_{t}^{\sz}\}_{t\geq 0}$ is the heat semigroup generated by
  $-(\Delta+\sz)$, that is $P_t^{\sz}=\e^{-t\sz}P_t$.  We write
  \begin{align*}
    Tb_i=TP_{t_i}^{\sz}b_i+T(I-P_{t_i}^{\sz})b_i,
  \end{align*}
  where $t_i=r_i^2$ with $r_i$ the radius of $B_i$.  By Lemma
  \ref{lem-le}, we have
  \begin{align*}
    \biggl\|\sum_iP_{t_i}^{\sz}b_i\biggr\|_2^2\ls\lambda \|f\|_1.
  \end{align*}
  This combined with the $L^2$-boundedness of $T$ yields
  \begin{align}\label{eqn-them1-1}
    \mu\left(\left\{x\in 2B\colon \lf|T\lf(\sum_i P_{t_i}^{\sz}b_i\r)(x)\r|>\frac{\lambda}{2}\right\}\right)
    \ls \frac{1}{\lambda}\|f\|_1
  \end{align}
  as desired. Consider now the term $T\sum_i(I-P_{t_i}^{\sz})b_i$. We
  write
  \begin{align}\label{eqn-them-4}
    &\mu\lf(\lf\{x\in 2B\colon\lf|T\lf(\sum_i(I-P_{t_i}^{\sz})b_i\r)(x)\r|>\frac{\lambda}2\r\}\r)\\ \notag
    &\hs\leq \sum_{i}\mu(2B_i)+\mu\lf(\lf\{x\in 2B\setminus\cup_i2B_i\colon \lf|T\lf(\sum_i(I-P_{t_i}^{\sz})b_i\r)\r|(x)>\frac{\lambda}{2}\r\}\r).
  \end{align}
  From Lemma \ref{lem-czdecom}, it follows that
  \begin{align}\label{eqn-them1-2}
    \sum_{i}\mu(2B_i)\ls\frac{\|f\|_1}{\lambda}
  \end{align}
  as desired. To estimate the second term, let $k_{t_i}^{\sz}(x,y)$
  denote the integral kernel of the operator $T(I-P_{t_i}^{\sz})$.
  Note that
  \begin{align*}
    (\Delta+\sz)^{-1}(I-P_{t_i}^{\sz})
    &=\int_0^{+\infty}(P_{s}^{\sz}-P_{t_i+s}^{\sz})\,ds=\int_0^{t_i}P_{s}^{\sz}\,ds
  \end{align*}
  and
  \begin{align*}
    T(I-P_{t_i}^{\sz})&=\Hess(\Delta+\sigma)^{-1}(I-P_{t_i}^{\sz})=\int_0^{t_i}\Hess P_{s}^{\sz}\,ds.
  \end{align*}
  Therefore,
  \begin{align}\label{eqn-kernelKt}
    k_{t_i}^{\sz}(x,y)=\int_0^{t_i}\Hess_xp_s^{\sz}(x,y)\,ds,
  \end{align}
  where $p_s^{\sigma}=\e^{-\sz s}p_s$ is the kernel of $P_s^{\sz}$
  with respect to $\mu$.  Since $b_i$ is supported in $B_i$, we have
  \begin{align}\label{eqn-them-3}
    \int_{2B\setminus (2B_i)}
    \lf|T\lf((I-P_{t_i}^{\sz})b_i\r)(x)\r|\,\mu(dx)
    &\leq 
      \int_{2B\setminus (2B_i)}\lf(\int_{B_i}|k_{t_i}^{\sz}(x,y)||b_i(y)|\,\mu(dy)\r)\,\mu(dx)\\ \notag
    &\leq \int_{B_i}\lf(\int_{\rho(x,y)\geq t_i^{1/2}}|k_{t_i}^{\sz}(x,y)|\,\mu(dx)\r)|b_i(y)|\,\mu(dy).
  \end{align}
  Now by means of \eqref{eqn-kernelKt} and Corollary \ref{prop1}, we
  get
  \begin{align*}
    \int_{\rho(x,y)\geq t_i^{1/2}}|k_{t_i}^{\sz}(x,y)|\,\mu(dx)
    &\leq 
      \int_0^{t_i}\lf(\int_{\rho(x,y)\geq t_i^{1/2}}|\Hess_xp_{s}(x,y)|\,\mu(dx)\r)\e^{-s\sigma}\,ds\\
    &\leq C\int_0^{t_i} \e^{-\beta t_i/s}s^{-1}\,\e^{C''s}\,\e^{-s\sigma}(1+\sqrt{s})\,ds\\
    &\leq C\int_0^1\frac{\e^{-\beta/u}}{u}\,du<\fz,
  \end{align*}
  where for the last inequality we use the fact that
$$\e^{s(C''-\sigma) }(1+\sqrt{s})<\fz,\quad s\in (0,\fz).$$ 
The estimate above together with \eqref{eqn-them-3} and Lemma
\ref{lem-czdecom} implies that
\begin{align}\label{eqn=them-6}
  \mu\lf(\lf\{x\in 2B\setminus\cup_i2B_i\colon
  \lf|T\lf(\sum_i(I-P_{t_i}^{\sz})b_i\r)(x)\r|>\frac{\lambda}{2}\r\}\r)
  \ls \frac{\|f\|_1}{\lz}.
\end{align}
Altogether, combining \eqref{eqn-them-8} through \eqref{eqn-them1-1},
\eqref{eqn-them1-2} and \eqref{eqn=them-6}, we conclude that
\eqref{eqn-localT-1} holds which completes the proof of Theorem
\ref{them-1}.
\end{proof}

\subsection{The case $p\in (2,\infty)$}\label{s3.2}
Let $M$ be a complete Riemannian manifold satisfying \eqref{eqn-H}.
In this subsection, we prove Theorem \ref{theorem-2} and show that the
Calder\'on-Zygmund inequality {\bf CZ}$(p)$ holds for all
$p\in(2,\fz)$, that is,
\begin{align}\label{eqn-Tpg2}
  \lf\|T f\r\|_{p}\ls \|f\|_p
\end{align}
holds for any $f\in L^p(M)$ with $T$ as in \eqref{eqn-defT} and
$\sigma\in (2K+\theta,\infty)$ where $K,\theta$ are respectively in
\eqref{eqn-Ric} and \eqref{R-esti}.

To this end, let $w$ be a $C^{\infty}$ function on $[0,\infty)$
satisfying $0\leq w\leq 1$ and
\begin{align*}
  w(t)=\begin{cases}1\ & \text{on} \  [0,3/4],\\
    0 \ & \text{on} \ [1,\fz),
  \end{cases}
\end{align*}
and let $\wz{T}$ be an operator defined by
\begin{align}\label{eqn-wzT}
  \wz{T}f:=\int_0^{\infty}v(t)\Hess P_tf\,dt
\end{align}
with $v(t):=w(t)\e^{-\sigma t}$. We need the following lemma, which
reduces \eqref{eqn-Tpg2} to a time and spatial localized version.

\begin{lemma}\label{lem-tpg2}
  Assume that \eqref{eqn-H} holds.  Let $p\in (2,\infty)$ and
  $\{x_j\}_{j\in \Lambda}$ be a countable subset of $M$ having the
  finite overlap property as in Lemma \textup{\ref{lem-fop}}. If there
  exists a positive constant $C$ such that
  \begin{align}\label{eqn-wTfb}
    \lf\| \, |\wz T (f)|\, \r\|_{L^p(B(x_j,4))}\le C \lf\|f\r\|_{L^p(B(x_j,1))}
  \end{align}
  holds for any $j\in \Lambda$ and $f\in C_0^\fz(B(x_j,1))$ with
  $\wz T$ defined as in \eqref{eqn-wzT}, then \eqref{eqn-Tpg2} holds.

\end{lemma}

\begin{proof}
  By the fact that $w\equiv1$ on $[0,3/4]$, we obtain from Proposition
  \ref{lem1} that if $\sigma> 2K+\theta$, then for any $g\in L^p(M)$,
  \begin{align*}
    \Big\|\int_0^{\infty}(1-w(t))\,\e^{-\sigma t}|\Hess P_{t}g|\,dt\Big\|_{p}\ls 
    \int_{3/4}^{\infty}\e^{(2K+\theta-\sigma)t}\, \frac{1+\sqrt{t}}{t}dt\, \|g\|_{p}\ls \|g\|_p.
  \end{align*}
  This and \eqref{eqn-wzT} imply that to prove \eqref{eqn-Tpg2}, it
  suffices to show that
  \begin{align}\label{eqn-wTfb-1}
    \lf\| \,|\wz T (g)|\, \r\|_{p}\ls \lf\|g\r\|_{p}
  \end{align}
  for any $g\in C^\fz_0(M)$.

  Let $(x_j)_{j\in \Lambda}$ be a countable subset of $M$ having the
  finite overlap property as in Lemma \ref{lem-fop}.  Let $\{\fai\}_j$
  be a corresponding $C^{\infty}$ partition of unity such that
  $0\leq \varphi_j\leq 1$ and $\fai_j$ is supported in
  $B_j:=B(x_j,1)$. Let $\chi_j$ be the characteristic function of the
  ball $4B_j$. For any $g\in C_0^{\infty}(M)$ and $x\in M$, write
  \begin{align}\label{eqn-wztq}
    \wz{T}g(x)\leq \sum_{j\in \Lambda} \chi_j\wz{T}(g\varphi_j)(x)+\sum_{j\in \Lambda}(1-\chi_j)\wz{T}(g\varphi_j)(x)=:{\rm I}(x)+{\rm II}(x).
  \end{align}
  By Lemma \ref{lem-fop}, we know
  \begin{align*}
    \sum_{j\in \Lambda}\lf|(1-\chi_j)(x)\varphi_j(y)\r|\leq N_0 \mathbbm{1}_{\{\rho(x,y)\geq 3\}}.
  \end{align*}
  Hence, by H\"{o}lder's inequality, we have
  \begin{align*}
    {\rm II}(x)
    &\leq \int_0^1 \int_M \lf|\Hess_xp_t(x,y)\r|\lf(\sum_{j\in \Lambda}\lf|(1-\chi_j)(x)\varphi_j(y)\r|\r)|g(y)|\,\mu(dy)\,dt\\
    &\leq N_0\int_0^1 \int_{\rho(x,y)\geq 3}\lf|\Hess_xp_t(x,y)||g(y)\r|\,\mu(dy)\,dt\\
    &\ls\int_0^1\left(\int_M\big|t\Hess_x
      p_t(x,y)\big|^p\,\e^{\gamma{\rho^2(x,y)}/{t}}
      \left(V(y,\!\sqrt{t})\right)^{p/p'}|g(y)|^p\,\mu(dy)\right)^{1/p}\frac{\e^{-{c}/{t}}}{t}\,
      dt,
  \end{align*}
  where $c={\gamma p'}/{p}$ and $\gz$ is a positive constant small
  enough so that in view of Proposition \ref{lem-hess-heat-kernel} and
  an argument similar to the proof of \eqref{eqn-ele1}, it implies
  that for any $t\in (0,1)$,
  \begin{align*}
    \int_M |t\Hess_xp_t(x,y)|^p\,\e^{\gamma\frac{\rho^2(x,y)}{t}}\,\mu(dx)\ls \frac{1}{\left(V(y,\!\sqrt{t})\right)^{p-1}}.
  \end{align*}
  This immediately implies
  \begin{align}\label{eqn-wztqp}
    &\int_M |{\rm II}(x)|^p\,\mu(dx)\\ \notag
    &\ls\int_M \left[\int_0^1\left(\int_M |t\Hess_x p_t(x,y)|^p\,\e^{\gamma{\rho^2(x,y)}/{t}}V(y,\!\sqrt{t})^{p/p'}|f(y)|^p\,\mu(dy)\right)^{1/p}
      \frac{\e^{-c/t}}{t}\,dt\right]^p\,\mu(dx)\\ \notag
    &\ls\int_M \int_0^1 \left(\int_M |t\Hess_x p_t(x,y)|^p\,\e^{\gamma{\rho^2(x,y)}/{t}}\,V(y,\!\sqrt{t})^{p/p'}|f(y)|^p\,\mu(dy)\right)\,dt\,\mu(dx)\\ \notag
    &\qquad\times\left(\int_0^1\frac{\e^{-{cp'}/{t}}}{t^{p'}}\,dt\right)^{p/p'}\\ \notag
    &\ls \int_0^1\lf[\int_M[V(y,\!\sqrt{t})]^{p/p'}|f(y)|^p\lf(\int_M |t\Hess_x p_t(x,y)|^p\,\e^{\gamma{\rho^2(x,y)}/{t}}\,\mu(dx)\r)\,\mu(dy)\r]\,dt \\ \notag
    &\ls \int_M |f(y)|^p\mu(dy)
  \end{align}
  as desired.

  Now we estimate $\mathrm{I}(x)$. Using Lemma \ref{lem-fop}, we know
  that the balls $\{4B_j\}_{j\in \Lambda}$ are of uniform overlap and
  hence
  \begin{align*}
    \sum_j \|h\chi_j\|_{p'}^{p'}\ls \|h\|_{p'}^{p'}
  \end{align*}
  for all $h\in C_0^\fz(M)$. Since $g\fai_j\in C_0^\fz (B(x_j,1))$ and
  using \eqref{eqn-wTfb}, we conclude that
  \begin{align*}
    \lf|\int_Mh(x)\mathrm{I}(x)\,\mu(dx)\r|
    &\le 
      \int_M|h(x)|\Big|\sum_j\chi_j\wz{T}(g\varphi_j)(x)\Big|\,\mu(dx)\\
    &\ls \sum_j \|g\varphi_j\|_p\|h\chi_j\|_{p'}\ls \|g\|_{p}\|h\|_{p'}, 
  \end{align*}
  which together with \eqref{eqn-wztq} and \eqref{eqn-wztqp} implies
  \eqref{eqn-Tpg2}, and hence finishes the proof of Lemma
  \ref{lem-tpg2}.
\end{proof}

From Lemma \ref{lem-tpg2}, it is easy to see that to prove $T$ is of
strong type $(p,p)$ for $p>2$, it suffices to show that $\wz T$ is
bounded from $L^p(B_j,\mu)$ to $L^p(4B_j,\mu)$ for each $j\in \Lambda$
as in \eqref{eqn-wTfb}. To this end, we need an $L^p$ local bounded
criterion from \cite{PTTS-2004} via maximal functions.  Recall that
the {\it local maximal function } by
\begin{align}\label{eqn-localMF}
  (\SM_{\rm loc}f)(x):=\sup_{\substack{B\ni x \\
  r(B)\leq 32}}\frac{1}{\mu(B)}\int_B |f|\,d\mu,\quad x\in M,
\end{align}
for any locally integrable function $f$ on $M$.  From \eqref{eqn-GD},
it follows that $\SM_{\rm loc}$ is bounded on $L^p(M)$ for all
$1<p\leq \infty$.  For a measurable subset $E\subset M$, the {\it
  maximal function relative to $E$ } is defined by
\begin{align}\label{eqn-relaMF}
  (\SM_Ef)(x):=\sup_{B\text{ ball in }  M,\,B\ni x} \frac{1}{\mu(B\cap E)}\int_{B\cap E}|f|\,d\mu,
  \quad x\in E,
\end{align}
for any locally integrable function $f$ on $M$. If in particular $E$
is a ball with radius $r$, it is enough to consider balls $B$ with
radii not exceeding $2r$. It is also easy to see $\SM_E$ is of weak
type $(1,1)$ and $L^p(M)$-bounded for $1<p\leq \infty$ if $E$
satisfies the {\it relative doubling property}, namely, if there
exists a constant $C_E$ such that for all $x\in E$ and $r>0$,
\begin{align}\label{RD}
  \mu(B(x,2r)\cap E)\leq C_E\,\mu(B(x,r)\cap E).
\end{align}
Note that in Lemma \ref{lem-fop} (iv), for any $j\in \Lambda$, the
subsets $4B_j$ satisfy the relative doubling property \eqref{RD} with
a constant independent of $j$.

The following theorem is essential to the proof of Theorem
\ref{theorem-2}.

\begin{lemma}\label{theorem-add-2}
  Let $p\in (2,\infty)$ and assume that \eqref{eqn-GD} holds. For any
  ball $B$ of $M$ centered in $4B_j$ with radius less than 8, assume
  that
  \begin{enumerate}[\rm(i)]
  \item there exists an integer $n$ depending only on condition
    \eqref{eqn-GD} such that the map
    $f\rightarrow \SM^{\#}_{4B_j,\tilde{T},n}f$ is bounded from
    $L^p(B_j,\mu)$ to $L^p(4B_j,\mu)$ with operator norm independent
    of $j$, where
    \begin{align*}
      \SM^{\#}_{4B_j,\tilde{T},n}f(x):=\dsup_{B \text{ \rm ball in }M,\, B\ni x}\lf(\frac{1}{\mu(B\cap 4B_j)}\dint_{B\cap 4B_j}
      \lf|\wz{T}(I-P_{r^2})^nf(y)\r|^2\,\mu(dy)\r)^{1/2};
    \end{align*}
  \item for all $k\in \{1,2,\ldots,n\}$, and all $f\in L^2(M,\mu)$
    supported in $B_j$, there exists a sublinear operator $S_j$
    bounded from $L^p(B_j,\mu)$ to $L^p(4B_j,\mu)$ with operator norm
    independent of $j$ such that for $x\in B\cap 4B_j$,
    \begin{align}\label{eq2}
      \left(\frac{1}{\mu(B\cap 4B_j)}\int_{B\cap 4B_j}|\wz{T}P_{kr^2}f|^p\,d\mu\right)^{1/p}
      \leq C\left(\SM_{4B_j}(|\wz{T}f|^2)+(S_jf)^2\right)^{1/2}(x)
    \end{align}
    where $\SM_{4B_j}$ is as in \eqref{eqn-relaMF}.
  \end{enumerate}
  Then $\wz T$ is bounded from $L^p(B_j,\mu)$ to $L^p(4B_j,\mu)$, that
  is \eqref{eqn-wTfb} holds with a constant depending on $p$, the
  doubling constant $C_D$ in \eqref{eqn-D}, the operator norm of
  $\wz{T}$ on $L^2(M,\mu)$, the operator norms of
  $\SM^{\#}_{4B_j,\tilde{T},n}$ and $S_j$ on $L^p$, and the constant
  in \eqref{eq2}.
\end{lemma}

\begin{proof}
  We use the following result \cite[Theorem 2.4]{PTTS-2004}:

  \noindent{\it Let
  $p_0\in (2,\infty]$ and $(M,\mu,\rho)$ be a measured metric  
  space. Suppose that $T$ is a bounded sublinear operator which is
  bounded on $L^2(M,\mu)$, and let $\{A_r\}_{r>0}$, be a family of
  linear operators acting on $L^2(M,\mu)$. Let $E_1$ and $E_2$ be two
  subsets of $M$ such that $E_2$ has the relative doubling property,
  $\mu(E_2)<\infty$ and $E_1\subset E_2$. Assume that
  \begin{enumerate}[\rm(i)]
  \item the sharp maximal functional $ \SM_{E_2,T,A}^{\#}$ 
    is bounded from $L^p(E_1,\mu)$ into $L^p(E_2,\mu)$ for all
    $p\in (2,p_0)$, where
    \begin{align*}
      &(\SM^{\#}_{E_2,T,A}f)^2(x)=\sup_{B \ \text{ball in}\ M,\ B\ni x}\frac{1}{\mu(B\cap E_2)}\int_{B\cap E_2}|T(I-A_{r(B)})f|^2\,d\mu
    \end{align*}
    for $x\in E_2$;
  \item for some sublinear operator $S$ bounded from $L^p(E_1,\mu)$
    into $L^p(E_2, \mu)$ for all $p\in (2,p_0)$,
    \begin{align}\label{TA-2}
      \left(\frac{1}{\mu(B\cap E_2)}\int_{B\cap E_2}|TA_{r(B)}f|^{p_0}\,d\mu\right)^{1/{p_0}}
      \leq C\left(\SM_{E_2}(|Tf|^2)+(S(f)^{2}\right)^{1/2}(x),
    \end{align}
    for all $f\in L^2(M,\mu)$ supported in $E_1$, all balls $B$ in $M$
    and all $x\in B\cap E_2$, where $r(B)$ is the radius of $B$.
  \end{enumerate}
  If $2<p<p_0$ and $Tf\in L^p(E_2,\mu)$ whenever $f\in L^p(E_1,\mu)$,
  then $T$ is bounded from $L^p(E_1,\mu)$ into $L^p(E_2,\mu)$ and its
  operator norm is bounded by a constant depending only on the
  operator norm of $T$ on $L^2(M,\mu)$, $C_{E_2}$ (see in \eqref{RD}
  for $E_2$), $p$ and $p_0$, the operator norms of
  $\SM^{\#}_{4B_j,\tilde{T},n}$ and $S$ on $L^p$, and the constant in
  \eqref{TA-2}.}

Here we may take $E_1$ and $E_2$ as $B_j$ and $4B_j$ respectively, as
the sets $B_j$ and $4B_j$ have the relatively volume doubling property
as in \eqref{RD} with the constant $C_E$ independent of $j$ (see Lemma
\ref{lem-fop}). Moreover, taking the operators $\{A_r\}_{r>0}$ such
that
\begin{align*}
  I-A_r=(I-P_{r^2})^n
\end{align*}
for some integer $n$ sufficiently large, then the result follows
directly.
\end{proof}

Now it suffices to check (i) and (ii) of Lemma \ref{theorem-add-2}. We
establish two technical lemmas which verify (i) and (ii)
respectively. To this end, observe that \eqref{eqn-GD} implies: for
all $r_0>0$ there exists $C_{r_0}$ such that for all $x\in M$,
$r\in (0,r_0)$,
\begin{align*}
  V(x,2r)\leq C_{r_0} V(x,r).
\end{align*}
An easy consequence of the definition is that for all $y\in M$,
$0<r<8$ and $s\geq 1$ satisfying $sr<32$,
\begin{align}\label{local-doubling}
  V(y,sr)\leq Cs^{D_{L}} V(y,r),
\end{align}
for some constants $C$ and $D_{L}>0$. The following lemma plays an
important role, when checking (i) of Lemma \ref{theorem-add-2}.

\begin{lemma}\label{lem-term1}
  Assume that \eqref{eqn-H} hold.  Then there exists an integer $n$
  such that the inequality
  \begin{align}\label{eq-Hess-1}
    \dsup_{B \, ball\ in\  M,\,B\ni x}\lf(\frac{1}{\mu(B\cap 4B_j)}\dint_{B\cap 4B_j}
    \lf|\wz{T}(I-P_{r^2})^nf(y)\r|^2\,\mu(dy)\r)^{1/2}
    \leq C \left(\SM_{\rm loc}(|f|^2)(x)\right)^{1/2}
  \end{align}
  holds for any $x\in 4B_j$, $f\in L^2(4B_j)$ satisfying
  $\supp f\subset 4B_j$, where $\SM_{\rm loc}$ is defined by
  \eqref{eqn-localMF}.
\end{lemma}

\begin{proof}
  Viewing the left-hand side of \eqref{eq-Hess-1} as the maximal
  function relative to $4B_j$, since the radius of $4B_j$ is $4$, it
  is enough to consider balls $B$ of radii not exceeding $8$.  Let
  $B=B(x_0,r)$ be an arbitrary ball containing $x$ and satisfying
  $x_0\in 4B_j$ and $r\in (0,8)$. By Lemma \ref{lem-fop}, we know that
  \begin{align}\label{B-control}
    \mu(B)\ls \mu(B\cap 4B_j)
  \end{align}
  with an implicit constant independent of $B$ and $j$. Hence, it is
  easy to see
  \begin{align*}
    \frac{1}{\mu(B\cap 4B_j)}\int_{B\cap 4B_j} \big|\wz{T}(I-P_{r^2})^nf\big|^2\,d\mu\ls \frac{1}{\mu(B)}\int_{B}\big|\wz{T}(I-P_{r^2})^nf\big|^2\,d\mu.
  \end{align*}
  Thus, we only need to show that
  \begin{align}\label{eqn-lem-term1-1}
    \dsup_{B\ \text{ball in}\ M,\, B\ni x}\lf(\frac{1}{\mu(B)}\dint_{B}
    \lf|\wz{T}(I-P_{r^2})^nf(y)\r|^2\,\mu(dy)\r)^{1/2}
    \ls  \left(\SM_{\rm loc}(|f|^2)(x)\right)^{1/2},
  \end{align}
  for any $x\in 4B_j$.  Since $r<8$, we choose $i_r\in \Z_+$
  satisfying
  \begin{align}\label{def-ir}
    2^{i_r}r\leq 8< 2^{i_r+1}r.
  \end{align}
  Denote by
  \begin{align}\label{C-B}
    C_i\ \text{ the annulus}\  2^{i+1}B\setminus (2^iB)\ \  \ \text{if}\  i\geq 2\, \text{
    and} \, C_1=4B. 
  \end{align}
  Using the fact $\supp f\subset 4B_j\subset 2^i B$ when $i>i_r$, we
  find that
  \begin{align*}
    f=\dsum_{i=1}^{i_r} f\mathbbm{1}_{C_i}=:\dsum_{i=1}^{i_r} f_i
  \end{align*}
  which then implies
  \begin{align}\label{eqn-lem-term1-2}
    \lf\|\,| \wz{T}(I-P_{r^2})^nf|\,\r\|_{L^2(B)}\leq \dsum_{i=1}^{i_r}\lf\|\,|\wz{T}(I-P_{r^2})^nf_i|\,\r\|_{L^2(B)}.
  \end{align}
  For $i=1$ we use the $L^2$-boundedness of $\wz{T}(I-P_{r^2})^n$ to
  obtain
  \begin{align}\label{eqn-lem-term1-3}
    \lf\|\,|\wz{T}(I-P_{r^2})^nf_1|\,\r\|_{L^2(B)}\leq \|f\|_{L^2(4B)}\leq \mu(4B)^{1/2}(\SM_{\rm loc}(|f|^2)(x))^{1/2}
  \end{align}
  as desired.  For $i\geq 2$, we infer from \eqref{eqn-wzT} that
  \begin{align*}
    \wz{T}(I-P_{r^2})^nf_i
    &=\int_0^{\infty}v(t)\Hess(P_t(I-P_{r^2})^nf_i)\,dt\\
    &=\dint_0^\fz v(t)\dsum_{k=0}^n \binom nk (-1)^k \Hess\,P_{t+kr^2}(f_i)\,dt \\ 
    &=\dint_0^\fz \lf(\dsum_{k=0}^n \binom nk (-1)^k \mathbf{1}_{\{t>kr^2\}} v(t-kr^2)\r)\Hess\, P_{t}(f_i)\,dt\\
    &=:\int_0^{\infty}g_r(t)\,\Hess P_tf_i\,dt.
  \end{align*}
  For $g_r$, according to the definition $v(t)=w(t)\e^{-\sigma t}$ and
  an elementary calculation, we observe that
  \begin{align*}
    \begin{cases}
      |g_r(t)|\ls 1 \quad  &\mbox{for}\ \  0<t \leq \ (1+nr^2)\wedge (1+n)r^2,\\
      |g_r(t)| \ls r^{2n} \quad  &\mbox{for}\ \   (1+nr^2)\wedge (1+n)r^2<t \leq 1+nr^2,\\
      g_r(t)=0\quad & \mbox{for}\ \ t > 1+nr^2,
    \end{cases}
  \end{align*}
  where the second estimate comes from $w^{(i)}(t)\ls 1$ for
  $t\in [0,1]$ and $i\in \{1,\ldots,n\}$, along with
  \begin{align*}
    \Bigg|\sum_{k=0}^n\binom nk (-1)^k w(t-kr^2) \Bigg|\leq C_n \sup_{u\geq \frac{t}{n+1}}\Big|(w(u)\,\e^{-\sigma u})^{(n)}\Big|r^{2n}\ls r^{2n}
  \end{align*}
  for all $t\geq (1+n)r^2$ and some constant $C_n$.  Combined with
  Lemma \ref{lem-Gaffney}, this gives
  \begin{align*}
    \lf\|\,\big|\wz{T}(I-P_{r^2})^nf_i\big|\,\r\|_{L^2(B)}\ls \left(\int_0^{\infty}|g_r(t)|(1+\sqrt{t})\,\e^{-{\alpha'4^ir^2}/{t}}\,\frac{ dt}{t}\right)\|f_i\|_{L^2(C_i)}
  \end{align*}
  where using the fact that $0<r<8$, we know
  \begin{align*}
    &\int_0^{\infty}(1+\sqrt{t})|g_r(t)|\,\e^{-\frac{\alpha'4^ir^2}{t}}\,\frac{dt}{t}\\
    &\ls  \int_{0}^{(1+nr^2)\wedge (1+n)r^2}(1+\sqrt{t})\,\e^{-\frac{\alpha' 4^i r^2}{t}}\frac{dt}{t}+\int_{(1+nr^2)\wedge (1+n)r^2}^{1+nr^2}(1+\sqrt{t})r^{2n}\,\e^{-\frac{\alpha' 4^i r^2}{t}}\frac{dt}{t} \\
    &\leq C_n \left(\int_0^{(n+1)r^2}\e^{-\frac{\alpha' 4^i r^2}{t}}\frac{dt}{t} +\int_{(1+nr^2)\wedge (1+n)r^2}^{1+nr^2}(1+\sqrt{t})r^{2n}\frac{t^{n-1}}{4^{in}r^{2n}}\, {dt}\right)\\ 
    &\leq C_n\left(4^{-in}+ 4^{-in}r^{2n}(1+\sqrt{r})\right)\leq C_n'4^{-in}. 
  \end{align*}
  Now an easy consequence of the local doubling
  \eqref{local-doubling}, since $r(2^iB)\leq 8$ when
  $1\leq i\leq i_r$, is that
  \begin{align*}
    \mu(2^{i+1}B)\leq C2^{(i+1)D_{L}}\mu(B),
  \end{align*}
  with constants $C$ and $D_{L}$ independent of $B$ and
  $i$. Therefore, as $C_i\subset 2^{i+1}B$,
  \begin{align*}
    \|f\|_{L^2(C_i)}\leq \mu(2^{i+1}B)^{1/2}(\SM_{\rm loc}(|f|^2)(x))^{1/2}\leq C2^{iD_{L}/2}\mu(B)^{1/2}(\SM_{\rm loc}(|f|^2)(x))^{1/2}.
  \end{align*}
  Using the definition of $i_r$, $r\leq 8$, and by choosing
  $2n>D_{L}/2$, we finally obtain
  \begin{align*}
    &\big\|\,|\wz{T}(I-P_{r^2})^nf|\,\big\|_{L^2(B)}\leq C'\left(\sum_{i=1}^{i_r}2^{i(D_{L}/2-2n)}\right)\mu(B)^{1/2}(\SM_{\rm loc}(|f|^2)(x))^{1/2},
  \end{align*}
  which proves Proposition \ref{lem1}.
\end{proof}

The following lemma is essential to the proof of part (ii) of Lemma
\ref{theorem-add-2}.

\begin{lemma}\label{lem-small-radii}
Let $p\in (2,\infty)$. Assume that \eqref{eqn-H} holds. For a ball
  $B$ with radius $r\in (0,8)$, let $i=i_r$ be an integer such that
  \eqref{def-ir} holds. Then the following estimate hold: For 
   any $C^2$-function $f$ supported in $C_i$ as in
  \eqref{C-B}, and each $k\in \{1,\ldots, n\}$, where
  $n\in \mathbb{N}$ is chosen according to Lemma
  \textup{\ref{lem-term1}}, one has
  \begin{align}
    &\left(\frac{1}{\mu(B)}\int_B|\Hess P_{kr^2}f|^p\,d\mu\right)^{1/p}\label{Hess-f-esti1}\\
    &\quad\leq C\e^{-\alpha_1 4^i}\left[\left(\frac{1}{\mu(2^{i+1}B)}\int_{C_i}|\nabla f|^2\,d\mu\right)^{1/2}+\left(\frac{1}{\mu(2^{i+1}B)}\int_{C_i}|\Hess f|^2\,d\mu\right)^{1/2}\right]\notag  \\
    \intertext{and}
    &\left(\frac{1}{\mu(B)}\int_B|\Hess P_{kr^2}f|^p\,d\mu\right)^{1/p}\leq   \frac{C\e^{-\alpha_2 4^i}}{r^2} \left(\frac{1}{\mu(2^{i+1}B)}\int_{C_i}| f|^2\,d\mu\right)^{1/2}  \label{Hess-f-esti2}   
  \end{align} 
  for some positive constants $C$ and $\alpha_1$, $\alpha_2$ depending
  on $K$,$\theta$, $p$ and the constants in \eqref{eqn-GD}.  
\end{lemma}

\begin{proof}
  We first observe that Lemma \ref{lem5} yields
  \begin{align}\label{Hess-ineq}
    \left(\int_B|\Hess P_{t}f|^p\,d\mu\right)^{1/p}
    & \leq \e^{2Kt}\left(\int_B \Big(P_t|\Hess f|^2\Big)^{p/2}(x)\,\mu(dx)\right)^{1/p}\notag\\
    &\quad +C \e^{(2K+\theta)t}\left(\int_B  \left(P_t|df|^2\right)^{p/2}(x)\,\mu(dx)\right)^{1/p}.
  \end{align}
  We substitute $t=kr^2$ in estimate~\eqref{Hess-ineq}. By Lemma
  \ref{lem-Gausspe}, one has the upper bound of $p_t(x,y)$,
  \begin{align*}
    p_t(x,y)\leq \frac{C_{\alpha}}{V(y,\!\sqrt{t})}\exp{\left(-\alpha\frac{\rho^2(x,y)}{t}+C_1Kt\right)}
  \end{align*}
  for all $x,y\in M$. Since $r\leq 8$, it follows from the above
  estimate that for all $x\in B$,
  \begin{align*}
    \left|P_{kr^2}(|df|^2)(x)\right|
    &\leq C\int_{C_i}V(y,\sqrt{k}r)^{-1}\exp\left(-\alpha\frac{\rho^2(x,y)}{kr^2}+C_1Kkr^2\right)|df|^2(y)\,\mu(dy)\\
    &\leq C\e^{-{\alpha}4^i/k}\int_{C_i}V(y,\sqrt{k}r)^{-1}|df|^2(y)\,\mu(dy).
  \end{align*}
  Moreover, since $y\in C_i$, we have
  $2^{i+1}B\subset B(y, 2^{i+2}r)$, and then by \eqref{eqn-GD}, we
  know that
  \begin{align*}
    \frac{1}{V(y,\sqrt{k}r)}\leq \frac{2^{d(i+2)}\e^{C2^{i+2}}}{V(y,2^{i+2}r)}\leq \frac{2^{d(i+2)}\e^{C2^{i+2}}}{\mu(2^{i+1}B)}.
  \end{align*}
  It then follows that
  \begin{align*}
    |P_{kr^2}|df|^2(x)|\leq C\e^{-c4^i}2^{d(i+2)}\e^{C2^{i+2}}\left(\frac{1}{\mu(2^{i+1}B)}\int_{C_i}|\nabla f|^2\,d\mu\right)
  \end{align*}
  for all $x\in B$, and there exists $\alpha_1<c$ such that
  \begin{align}\label{gradient-ineq-esti}
    \left(\frac{1}{\mu(B)}\int_{B}(P_{kr^2}|\nabla f|^2)^{p/2}\,d\mu\right)^{1/p}\leq C\e^{-\alpha_1 4^i}\left(\frac{1}{\mu(2^{i+1}B)}\int_{C_i}|\nabla f|^2\,d\mu\right)^{1/2}.
  \end{align}
  With similar arguments, we obtain
  \begin{align*}
    \left(\frac{1}{\mu(B)}\int_{B}(P_{kr^2}|\Hess f|^2)^{p/2}\,d\mu\right)^{1/p}\leq C\e^{-\alpha_1 4^i}\left(\frac{1}{\mu(2^{i+1}B)}\int_{C_i}|\Hess f|^2\,d\mu\right)^{1/2}.
  \end{align*}
  Altogether these yields
  \begin{align*}
    &\left(\frac{1}{\mu(B)}\int_B|\Hess P_{kr^2}f|^p\,d\mu\right)^{1/p}\\
    &\leq C \e^{-\alpha_1 4^i}\left(\left(\frac{1}{\mu(2^{i+1}B)}\int_{C_i}|\nabla f|^2\,d\mu\right)^{1/2}+\left(\frac{1}{\mu(2^{i+1}B)}\int_{C_i}|\Hess f|^2\,d\mu\right)^{1/2}\right),
  \end{align*}
  which completes the proof of \eqref{Hess-f-esti1}.
  
Next we observe from Proposition \ref{lem1} that
\begin{align}\label{Hess-ineq2}
      \left(\int_B|\Hess P_{t}f|^p\,d\mu\right)^{1/p}
    &\leq\frac{(1+\sqrt{t})}{t}\e^{(2K+\theta)t}\left(\int_B \Big(P_t|f|^2\Big)^{p/2}(x)\,\mu(dx)\right)^{1/p}.
    \end{align}
    Substituting $t=kr^2$ in \eqref{Hess-ineq2}, as $r\in (0,8)$,
    there exists a constant $C$ such that
    \begin{align*}
      \left(\int_B|\Hess P_{kr^2}f|^p\,d\mu\right)^{1/p}
      & \leq \frac{C}{r^2}\left(\int_B \Big(P_{kr^2}|f|^2\Big)^{p/2}(x)\,\mu(dx)\right)^{1/p}.
    \end{align*}
    With similar arguments as for \eqref{gradient-ineq-esti}, we then
    complete the proof of \eqref{Hess-f-esti2}.\qedhere
\end{proof}

With the help of the Lemmas \ref{theorem-add-2}, \ref{lem-term1} and \ref{lem-small-radii},
we are now in position to turn to the proof of Theorem \ref{theorem-2}.

\begin{proof}[Proof of Theorem \ref{theorem-2}]

  By Lemma \ref{theorem-add-2}, we only need to show (i) and (ii) of
  Lemma \ref{theorem-add-2} hold true under our condition
  \eqref{eqn-H}. We first verify (i) of Lemma
  \ref{theorem-add-2}. Observe from Lemma \ref{lem-term1} that there
  is an integer $n$ depending only on $D_{L}$ as in
  \eqref{local-doubling} such that for all $f\in L^p(B_j)$ and
  $x\in 4B_j$,
  \begin{align*}
    \sup_{ B \,\text{ball in}\, M, B\ni x} \frac{1}{\mu(B\cap 4B_j)}\int_{B\cap 4B_j} |\wz{T}(I-P_{r^2})^nf|^2\,d\mu\leq C\SM_{\rm loc}(|f|^2)(x).
  \end{align*}
  Recall that $\SM_{\rm loc}$ is bounded on $L^p(M)$ for
  $1<p\leq \infty$; thus for all $p>2$, the operator
  $\SM^{\#}_{4B_j,\tilde{T},n}$ is bounded from $L^p(B_j,\mu)$ to
  $L^p(4B_j,\mu)$ uniformly in $j$, i.e.  assertion (i) is proved.

  Next, we prove (ii) of Lemma \ref{theorem-add-2}.  Assume that
  $f\in L^2(B_j)$ and let $h=\int_0^{\infty}v(t)P_tf\,dt$ with $v$
  defined as in \eqref{eqn-wzT}. Since $\wz{T}f=\Hess\, h$ and the
  inequality \eqref{B-control} for $B\cap 4B_j$, we have
  \begin{align*}
    \left(\frac{1}{\mu(B\cap 4B_j)}\int_{B\cap 4B_j}|\wz T P_{kr^2}f|^p\,d\mu\right)^{1/p}
    &= \left(\frac{1}{\mu(B\cap 4B_j)}\int_{B\cap 4B_j}|\Hess P_{kr^2}h|^p\,d\mu\right)^{1/p}\\
    &\ls \left(\frac{1}{\mu(B)}\int_B |\Hess P_{kr^2}h|^p\,d\mu\right)^{1/p}.
  \end{align*}
  We write
  \begin{align*}
    \Hess P_{kr^2}h=\sum_{i=1}^{\infty}\Hess P_{kr^2}(h\mathbbm{1}_{C_i})=\sum_{i=1}^{\infty}\Hess P_{kr^2}g_i,
  \end{align*} 
  where $g_i=h\mathbbm{1}_{C_i}$ and $C_i$ is as in \eqref{C-B}. Next,
  we distinguish the two regimes $i\leq i_r$ and $i>i_r$ where $i_r$
  is defined as in \eqref{def-ir}. In the regime $i\leq i_r$, by the
  inequality \eqref{Hess-f-esti1} in Lemma \ref{lem-small-radii}, we
  have
  \begin{align*}
    &\left(\frac{1}{\mu(B)}\int_B|\Hess P_{kr^2}g_i|^p\,d\mu\right)^{1/p}\\
    &\leq C\e^{-\alpha_1 4^i}\left(\left(\frac{1}{\mu(2^{i+1}B)}\int_{C_i}|\nabla h|^2\,d\mu\right)^{1/2}+\left(\frac{1}{\mu(2^{i+1}B)}\int_{C_i}|\Hess\, h|^2\,d\mu\right)^{1/2}\right)\\
    &\leq  C\e^{-\alpha_1 4^i}\left(\left(\SM_{\rm loc}(|\nabla h|^2)(x)\right)^{1/2}+\left(\frac{1}{\mu(2^{i+1}B)}\int_{C_i}|\Hess\, h|^2\,d\mu\right)^{1/2}\right).
  \end{align*}
  On the other hand, since $C_i\subset 2^{i+1}B$, we know
  \begin{align*}
    &\frac{1}{\mu(2^{i+1}B)}\int_{C_i}\big|\Hess\, h\big|^2\,d\mu\\
    &\leq \frac{1}{\mu(2^{i+1}B\cap 4B_j)}\int_{(2^{i+1}B)\cap 4B_j}|\Hess\, h|^2d\mu+\frac{1}{\mu(2^{i+1}B)}\int_{2^{i+1}B}\mathbbm{1}_{M\setminus 4B_j}|\Hess\, h|^2\,d\mu\\
    &\leq \SM_{4B_j}(|\Hess\, h|^2)(x)+\SM_{\rm loc}(|\Hess\, h|^2\mathbbm{1}_{M\setminus 4B_j})(x),
  \end{align*}
  for any $x\in B\cap 4B_j$. Hence in this case
  \begin{align}\label{part-1}
    &\left(\frac{1}{\mu(B)}\int_B \big|\sum_{i=1}^{i_r}\Hess P_{kr^2}g_i\big|^p\,d\mu\right)^{1/p} \\
    &\leq \sum_{i=1}^{ i_r}C\e^{-\alpha 4^i} \Big(\SM_{\rm loc}(|\nabla h|^2)+\SM_{4B_j}(|\Hess\,h|^2)+\SM_{\rm loc}(|\Hess\,h|^2\mathbbm{1}_{M\setminus 4B_j})\Big)^{1/2}(x).\notag
  \end{align}
  For the second regime $i>i_r$, we proceed with inequality \eqref{Hess-f-esti1} in Lemma \ref{lem-small-radii} that there exist 
    positive constants $c_1$ and $c_2$ such that   
  \begin{align}\label{Hess-gi}
    \left(\frac{1}{\mu(B)}\int_B|\Hess P_{kr^2}g_i|^p\,d\mu\right)^{1/p}\leq \frac{c_1\e^{-c_2 4^i}}{r^2}\left(\frac{1}{\mu(2^{i+1}B)}\int_{C_i}|h|^2\,d\mu\right)^{1/2}.
  \end{align}
  On the other hand, since $i>i_r$, it is easy to see
  $4B_j\subset2^{i+1}B$, thus
  \begin{align}\label{h-esti}
    \Big(\frac{1}{\mu(2^{i+1}B)}\int_{C_i}|h|^2\,d\mu\Big)^{1/2}&\leq \left(\frac{1}{\mu(2^{i+1}B)}\int_0^1 v(t)\int_{C_i} P_t|f|^2\, d\mu \, dt \right)^{1/2}\\
                                                                &\leq C\left(\frac{1}{\mu(4B_j)}\int_{B_j}|f|^2\,d\mu\right)^{1/2}=C\left(\frac{1}{\mu(4B_j)}\int_{B_j}|f|^2\,d\mu\right)^{1/2} \notag\\
                                                                &\leq C\big(\SM_{4B_j}(|f|^2)(x)\big)^{1/2}. \notag
  \end{align}
  The contribution of the terms in the second regime $i>i_r$ is
  bounded by combining \eqref{Hess-gi} and \eqref{h-esti},
  \begin{align}\label{part-2}
    \sum_{i>i_r}\left(\frac{1}{\mu(B)}\int_B|\Hess P_{kr^2}g_i|^p\,d\mu\right)^{1/p}\leq \sum_{i>i_r} \frac{c_1\e^{-c_2 4^i}}{r^2}\big(\SM_{4B_j}(|f|^2)(x)\big)^{1/2}
  \end{align}
  and it remains to recall that $1/r^2\leq 4^i/8$ and
  $1/r\leq 2^i/{(2\sqrt{2})}$ when $i>i_r$.

  We then conclude from \eqref{part-1} and \eqref{part-2} that for any
  $p>2$ and $k\in \{1,2,\ldots,n\}$, there exists a constant $C$
  independent of $j$ such that
  \begin{align*}
    \left(\frac{1}{\mu(B\cap 4B_j)}\int_{B\cap 4B_j}|\wz{T}P_{kr^2}f|^p\,d\mu\right)^{1/p}\leq C(\SM_{4B_j}(|\wz{T}f|^2)+(S_jf)^2)^{1/2}(x)
  \end{align*}
  all $f\in L^2(M,\mu)$ supported in $B_j$, all balls $B$ in $M$ and
  all $x\in B\cap 4B_j$, where the radius $r$ of $B$ is less than 8,
  and where
  \begin{align}\label{eqn-squref}
    (S_jf)^2:=\SM_{\rm loc}(|\wz{T}f|^2\mathbbm{1}_{M\setminus 4B_j})
    +\SM_{\rm loc}(|\nabla h|^2)+\SM_{4B_j}(|f|^2).
  \end{align}

  Our last step is to show that the operator $S_j$ defined in
  \eqref{eqn-squref} is bounded from $L^p(B_j)$ to $L^p(4B_j)$ for any
  $p\in (2,\fz)$ with operator norm independent of $j$.  By
  \eqref{eqn-squref}, we only need to show that the operators
$$\big(\SM_{\rm loc}(|\wz{T}f|^2\mathbbm{1}_{M\setminus 4B_j})\big)^{1/2},\ \
\big(\SM_{\rm loc}(|\nabla h|^2)\big)^{1/2} \ \mbox{ and }\
\big(\SM_{4B_j}(|f|^2)\big)^{1/2}$$ are respectively bounded from
$L^p(B_j)$ to $L^p(4B_j)$.  Indeed, for any $f\in L^p(4B_j)$, by Lemma
\ref{lem-fop} we know that $4B_j$ satisfies the doubling property
\eqref{eqn-D}, which combined with $p>2$ implies that
$\big(\SM_{4B_j}(|f|^2)\big)^{1/2}$ is bounded from $L^p(B_j)$ to
$L^p(4B_j)$ by a constant depending only on the doubling property
\eqref{eqn-D}.  On the other hand, using the fact from Bismut's
formula \cite{Bismut} that
\begin{align*}
  \lf\| \sqrt t\nabla P_t f \r\|_{p} \ls \e^{Kt} \|f\|_p,
\end{align*}
we deduce from $p>2$ and $\sigma>2K+\theta$ that
\begin{align*}
  \big\|\lf(\SM_{\rm loc}(|\nabla h|^2)\r)^{1/2}\big\|_p
  \ls \|\nabla h\|_{p}\leq \left(\int_0^{\infty}\frac{v(t)\,\e^{ Kt}}{\sqrt{t}}\,dt\right)\,\|f\|_{L^p(B_j)}\leq C\|f\|_{L^p(B_j)}.
\end{align*}
Finally, the $L^p$-boundedness of
$$\big(\SM_{\rm loc}(|\wz{T}f|^2\mathbbm{1}_{M\setminus
  4B_j})\big)^{1/2}$$ follows from the $L^{p/2}$-boundedness of
$\big(\SM_{\rm loc}(|\wz{T}f|^2\big)^{1/2}$ and an argument similar to
\eqref{eqn-wztqp}. This implies that the operator $S_j$ is bounded
from $L^p(B_j)$ to $L^p(4B_j)$ with an upper bound independent of $j$.

We conclude that the requirements (i) and (ii) in Lemma
\ref{theorem-add-2} both hold under the condition \eqref{eqn-H}.
Thus, the operator $\wz{T}$ is bounded from $L^p(B_j,\mu)$ to
$L^p(4B_j,\mu)$ for $p>2$ with a constant independent of~$j$.
Therefore, by Lemma \ref{lem-tpg2}, the operator $T$ is strong type
$(p,p)$ for $p>2$. This proves Theorem~\ref{theorem-2}.
\end{proof}

\section{Appendix}

In this appendix, we include the proof 
of {\bf (CZ)}$(2)$ in the case that the underlying manifold $M$ has a lower bound Ricci curvature 
for reader's convenience (see \cite{GP-15, Pigola}). 
To be precise, let $M$ be a complete Riemannian manifold satisfying the curvature condition
\eqref{eqn-Ric}.  For any
$u\in C_0^{\infty}(M)$, Bochner's formula gives
\begin{align}\label{eqn-BI}
  -\frac{1}{2}\Delta |\nabla u|^2=|\Hess u|_\HS^2-g(\nabla \Delta u, \nabla u)+\Ric(\nabla u,\nabla u).
\end{align}
Integration by parts, together with the lower Ricci bound, leads to
\begin{align*}
  0&=-\frac{1}{2}\int_M \Delta |\nabla u|^2\\
   &=\int_M |\Hess u|_\HS^2-\int_M g(\nabla \Delta u, \nabla u)+\int_M \Ric(\nabla u, \nabla u)\\
   &\geq \int_M |\Hess u|_\HS^2-\int_M(\Delta u)^2-K\int_M |\nabla u|^2\\
   &=\int_M |\Hess u|_\HS^2-\int_M(\Delta u)^2-K\int_M u\Delta u.
\end{align*}
Then, by estimating the last integral via Young's inequality, we
obtain for any $\epz>0$,
\begin{align}\label{T2-2}
  \|\,|\Hess u|\,\|_{L^2}^2\leq\|\,|\Hess u|_\HS\,\|_{L^2}^2\leq \frac{K\epz^2}{2}\|u\|_{L^2}^2+\lf(1+\frac{K}{2\epz^2}\r)\|\Delta u\|_{L^2}^2,
\end{align}
which establishes {\bf CZ(2)}.

\begin{remark}
  Inequality \eqref{T2-2} extends from $u\in C_0^{\infty}(M)$ to
  $u\in H^{2,2}(M)$.  Thus, in particular, if $u\in L^2(M)$ is a
  distributional solution to the Poisson equation
  \begin{align}\label{eqn-Poisson}
    \Delta u=f \quad \text{on }M
  \end{align}
  for some $f\in L^2(M)$, then \eqref{T2-2} provides an $L^2$-Hessian
  estimate of the solution $u$.
\end{remark}
 
\begin{remark}
  Recall that for $L^2$-gradient estimates of $u$ in
  \eqref{eqn-Poisson}, the lower bound on the Ricci tensor is not
  needed.  Indeed, integrating \eqref{eqn-Poisson} by parts and using
  Young inequality, one obtains directly for every $\epz>0$,
  \begin{align*}
    \int_M|\nabla u|^2=\int_Mf u
    \leq \frac{\epz^2}{2}\int_M u^2+\frac{1}{2\epz^2}\int_M f^2,
  \end{align*}
  which is the $L^2$-gradient estimate
  \begin{align}\label{eqn-Lge}
    \|\nabla u\|^2_{L^2}\leq \frac{\epz^2}{2}\|u\|_{L^2}^2+\frac{1}{2\epz^2}\|f\|_{L^2}^2.
  \end{align}
\end{remark}

\bibliographystyle{amsplain}%

\bibliography{C-Z-inequality}

\providecommand{\bysame}{\leavevmode\hbox to3em{\hrulefill}\thinspace}
\providecommand{\MR}{\relax\ifhmode\unskip\space\fi MR }
\providecommand{\MRhref}[2]{%
  \href{http://www.ams.org/mathscinet-getitem?mr=#1}{#2}
}
\providecommand{\href}[2]{#2}
\begin{thebibliography}{10}

\bibitem{AS-82}
Michael Aizenman and Barry Simon, \emph{Brownian motion and {H}arnack
  inequality for {S}chr\"{o}dinger operators}, Comm. Pure Appl. Math.
  \textbf{35} (1982), no.~2, 209--273. \MR{644024}

\bibitem{APT-03}
Marc Arnaudon, Holger Plank, and Anton Thalmaier, \emph{A {B}ismut type formula
  for the {H}essian of heat semigroups}, C. R. Math. Acad. Sci. Paris
  \textbf{336} (2003), no.~8, 661--666. \MR{1988128}

\bibitem{PTTS-2004}
Pascal Auscher, Thierry Coulhon, Xuan~Thinh Duong, and Steve Hofmann,
  \emph{Riesz transform on manifolds and heat kernel regularity}, Ann. Sci.
  \'{E}cole Norm. Sup. (4) \textbf{37} (2004), no.~6, 911--957. \MR{2119242}

\bibitem{BDG-21}
Robert Baumgarth, Baptiste Devyver, and Batu G\"{u}neysu, \emph{Estimates for
  the covariant derivative of the heat semigroup on differential forms, and
  covariant {R}iesz transforms}, arXiv:2107.00311 (2021).

\bibitem{Bismut}
Jean-Michel Bismut, \emph{Large deviations and the {M}alliavin calculus},
  Progress in Mathematics, vol.~45, Birkh\"{a}user Boston, Inc., Boston, MA,
  1984. \MR{755001}

\bibitem{CZ-52}
Alberto~P. Cald\'{e}ron and Antoni~S. Zygmund, \emph{On the existence of
  certain singular integrals}, Acta Math. \textbf{88} (1952), 85--139.
  \MR{52553}

\bibitem{Carron}
Gilles Carron, \emph{Riesz transform on manifolds with quadratic curvature
  decay}, Rev. Mat. Iberoam. \textbf{33} (2017), no.~3, 749--788. \MR{3713030}

\bibitem{Chavel}
Isaac Chavel, \emph{Eigenvalues in {R}iemannian geometry}, Pure and Applied
  Mathematics, vol. 115, Academic Press, Inc., Orlando, FL, 1984, Including a
  chapter by Burton Randol, With an appendix by Jozef Dodziuk. \MR{768584}

\bibitem{CLW-21}
Xin Chen, Xue-Mei Li, and Bo~Wu, \emph{Logarithmic heat kernels: estimates
  without curvature restrictions}, arXiv:2106.02746 (2021).

\bibitem{CH97}
Bennett Chow and Richard~S. Hamilton, \emph{Constrained and linear {H}arnack
  inequalities for parabolic equations}, Invent. Math. \textbf{129} (1997),
  no.~2, 213--238. \MR{1465325}

\bibitem{CW-book}
Ronald~R. Coifman and Guido Weiss, \emph{Analyse harmonique non-commutative sur
  certains espaces homog\`enes}, Lecture Notes in Mathematics, Vol. 242,
  Springer-Verlag, Berlin-New York, 1971, \'{E}tude de certaines int\'{e}grales
  singuli\`eres. \MR{0499948}

\bibitem{TX-99}
Thierry Coulhon and Xuan~Thinh Duong, \emph{Riesz transforms for {$1\leq p\leq
  2$}}, Trans. Amer. Math. Soc. \textbf{351} (1999), no.~3, 1151--1169.
  \MR{1458299}

\bibitem{CD-01}
\bysame, \emph{Riesz transforms for {$p>2$}}, C. R. Acad. Sci. Paris S\'{e}r. I
  Math. \textbf{332} (2001), no.~11, 975--980. \MR{1838122}

\bibitem{CD-03}
\bysame, \emph{Riesz transform and related inequalities on noncompact
  {R}iemannian manifolds}, Comm. Pure Appl. Math. \textbf{56} (2003), no.~12,
  1728--1751. \MR{2001444}

\bibitem{Davies93}
E.~Brian Davies, \emph{The state of the art for heat kernel bounds on
  negatively curved manifolds}, Bull. London Math. Soc. \textbf{25} (1993),
  no.~3, 289--292. \MR{1209255}

\bibitem{Davies97}
\bysame, \emph{Non-{G}aussian aspects of heat kernel behaviour}, J. London
  Math. Soc. (2) \textbf{55} (1997), no.~1, 105--125. \MR{1423289}

\bibitem{DR-96}
Xuan~T. Duong and Derek~W. Robinson, \emph{Semigroup kernels, {P}oisson bounds,
  and holomorphic functional calculus}, J. Funct. Anal. \textbf{142} (1996),
  no.~1, 89--128. \MR{1419418}

\bibitem{EL94}
K.~David Elworthy and Xue-Mei Li, \emph{Formulae for the derivatives of heat
  semigroups}, J. Funct. Anal. \textbf{125} (1994), no.~1, 252--286.
  \MR{1297021}

\bibitem{EL-98}
\bysame, \emph{Bismut type formulae for differential forms}, C. R. Acad. Sci.
  Paris S\'{e}r. I Math. \textbf{327} (1998), no.~1, 87--92. \MR{1650216}

\bibitem{GTbook}
David Gilbarg and Neil~S. Trudinger, \emph{Elliptic partial differential
  equations of second order}, Classics in Mathematics, Springer-Verlag, Berlin,
  2001, Reprint of the 1998 edition.

\bibitem{Gr97}
Alexander Grigor’yan, \emph{Gaussian upper bounds for the heat kernel on
  arbitrary manifolds}, J. Differential Geom. \textbf{45} (1997), no.~1,
  33--52. \MR{1443330}

\bibitem{Gu-12}
Batu G\"{u}neysu, \emph{Nonrelativistic hydrogen type stability problems on
  nonparabolic 3-manifolds}, Ann. Henri Poincar\'{e} \textbf{13} (2012), no.~7,
  1557--1573. \MR{2982633}

\bibitem{Gueneysu-book-2017}
\bysame, \emph{Covariant {S}chr\"{o}dinger semigroups on {R}iemannian
  manifolds}, Operator Theory: Advances and Applications, vol. 264,
  Birkh\"{a}user/Springer, Cham, 2017. \MR{3751359}

\bibitem{Guen-2017}
\bysame, \emph{Heat kernels in the context of {K}ato potentials on arbitrary
  manifolds}, Potential Anal. \textbf{46} (2017), no.~1, 119--134. \MR{3595965}

\bibitem{GP-15b}
Batu G\"{u}neysu and Diego Pallara, \emph{Functions with bounded variation on a
  class of {R}iemannian manifolds with {R}icci curvature unbounded from below},
  Math. Ann. \textbf{363} (2015), no.~3-4, 1307--1331. \MR{3412360}

\bibitem{GP-15}
Batu G\"{u}neysu and Stefano Pigola, \emph{The {C}alder\'{o}n-{Z}ygmund
  inequality and {S}obolev spaces on noncompact {R}iemannian manifolds}, Adv.
  Math. \textbf{281} (2015), 353--393. \MR{3366843}

\bibitem{GP-2019}
\bysame, \emph{{$L^p$}-interpolation inequalities and global {S}obolev
  regularity results}, Ann. Mat. Pura Appl. (4) \textbf{198} (2019), no.~1,
  83--96, With an appendix by Ognjen Milatovic. \MR{3918620}

\bibitem{GP13}
Batu G\"{u}neysu and Olaf Post, \emph{Path integrals and the essential
  self-adjointness of differential operators on noncompact manifolds}, Math. Z.
  \textbf{275} (2013), no.~1-2, 331--348. \MR{3101810}

\bibitem{HM03}
Steve Hofmann and Jos\'{e}~Mar\'{\i}a Martell, \emph{{$L^p$} bounds for {R}iesz
  transforms and square roots associated to second order elliptic operators},
  Publ. Mat. \textbf{47} (2003), no.~2, 497--515. \MR{2006497}

\bibitem{Veronelli_et_al-2021}
Shouhei Honda, Luciano Mari, Michele Rimoldi, and Giona Veronelli,
  \emph{Density and non-density of {$C_c^\infty\hookrightarrow W^{k,p}$} on
  complete manifolds with curvature bounds}, Nonlinear Anal. \textbf{211}
  (2021), Paper No. 112429, 26. \MR{4268755}

\bibitem{Kato}
Tosio Kato, \emph{Schr\"{o}dinger operators with singular potentials}, Israel
  J. Math. \textbf{13} (1972), 135--148 (1973). \MR{333833}

\bibitem{Li}
Xue-Mei Li, \emph{Stochastic differential equations on noncompact manifolds},
  University of Warwick, Thesis (1992).

\bibitem{Marini-Veronelli-2021}
Ludovico Marini and Giona Veronelli, \emph{The {$L^p$}-{C}alder\'{o}n-{Z}ygmund
  inequality on non-compact manifolds of positive curvature}, Ann. Global Anal.
  Geom. \textbf{60} (2021), no.~2, 253--267. \MR{4291611}

\bibitem{Ni}
Lei Ni, \emph{The {P}oisson equation and {H}ermitian-{E}instein metrics on
  holomorphic vector bundles over complete noncompact {K}\"{a}hler manifolds},
  Indiana Univ. Math. J. \textbf{51} (2002), no.~3, 679--704. \MR{1911050}

\bibitem{NST}
Lei Ni, Yuguang Shi, and Luen-Fai Tam, \emph{Poisson equation,
  {P}oincar\'{e}-{L}elong equation and curvature decay on complete {K}\"{a}hler
  manifolds}, J. Differential Geom. \textbf{57} (2001), no.~2, 339--388.
  \MR{1879230}

\bibitem{Pigola}
Stefano Pigola, \emph{Global {C}alder\'{o}n-{Z}ygmund inequalities on complete
  {R}iemannian manifolds}, arXiv:2011.03220v1 (2020).

\bibitem{rose_stollmann_2020}
Christian Rose and Peter Stollmann, \emph{Manifolds with {R}icci curvature in
  the {K}ato class: Heat kernel bounds and applications}, Analysis and Geometry
  on Graphs and Manifolds, London Mathematical Society Lecture Note Series,
  vol. 461, Cambridge University Press, 2020, pp.~76--94.

\bibitem{Simon-84}
Barry Simon, \emph{Schr\"{o}dinger semigroups}, Bull. Amer. Math. Soc. (N.S.)
  \textbf{7} (1982), no.~3, 447--526. \MR{670130}

\bibitem{Stroock}
Daniel~W. Stroock, \emph{An estimate on the {H}essian of the heat kernel},
  It\^{o}'s stochastic calculus and probability theory, Springer, Tokyo, 1996,
  pp.~355--371. \MR{1439536}

\bibitem{ThW-04}
Anton Thalmaier and Feng-Yu Wang, \emph{Derivative estimates of semigroups and
  {R}iesz transforms on vector bundles}, Potential Anal. \textbf{20} (2004),
  no.~2, 105--123. \MR{2032944}

\bibitem{Wang2004}
Caitlin Wang, \emph{The {C}alder\'{o}n-{Z}ygmund inequality on a compact
  {R}iemannian manifold}, Pacific J. Math. \textbf{217} (2004), no.~1,
  181--200. \MR{2105773}

\bibitem{Wbook14}
Feng-Yu Wang, \emph{Analysis for diffusion processes on {R}iemannian
  manifolds}, Advanced Series on Statistical Science \& Applied Probability,
  vol.~18, World Scientific Publishing Co. Pte. Ltd., Hackensack, NJ, 2014.
  \MR{3154951}

\end{thebibliography}
\end{document}